\documentclass[final]{siamart190516}
\usepackage{amsmath,amssymb,graphicx,float,url,enumitem,xcolor,color,soul,subcaption}
\usepackage{placeins}
\usepackage[margin=1in]{geometry}

\newcommand{\norm}[1]{\left\lVert #1 \right\rVert}
\newcommand{\inner}[2]{\left\langle #1,#2 \right\rangle}
\newcommand{\abs}[1]{\left|#1\right|}
\newcommand{\un}[0]{\boldsymbol{\nu}}
\newcommand{\normal}[1]{\frac{\partial #1}{\partial \un}}

\renewcommand{\div}[0]{\textnormal{div}\,}
\newcommand{\curl}[0]{\textnormal{\textbf{curl}}\,}
\newcommand{\grad}[0]{\textnormal{\textbf{grad}}\,}
\newcommand{\Div}[0]{\textnormal{Div}}
\newcommand{\Curl}[0]{\textnormal{Curl}}
\newcommand{\Hcurl}[1]{\mathbf{H}(\textnormal{\textbf{curl}},#1)}
\newcommand{\Hcurltr}[1]{\mathbf{H}_0(\textnormal{\textbf{curl}},#1)}
\newcommand{\Hcurlloc}[1]{\mathbf{H}_{\textnormal{loc}}(\textnormal{\textbf{curl}},#1)}
\newcommand{\Hcurlinc}[1]{\mathbf{H}_{\textnormal{inc}}(\textnormal{\textbf{curl}},#1)}
\newcommand{\Hdiv}[1]{\mathbf{H}(\textnormal{div},#1)}
\newcommand{\Hdivzero}[1]{\mathbf{H}(\textnormal{div}^0,#1)}
\newcommand{\Hdivnormzero}[1]{\mathbf{H}_0(\textnormal{div}^0,#1)}

\newcommand{\boldPsi}[0]{\boldsymbol{\Psi}}

\newcommand*\myat{{\fontfamily{ptm}\selectfont @}}
\renewcommand{\Im}[0]{\textnormal{Im}}
\renewcommand{\Re}[0]{\textnormal{Re}}

\newcommand{\proofend}{\qed\hfill $\Box$}

\usepackage{empheq}
\usepackage[many]{tcolorbox}

\tcbset{
  highlight math style={
    colback=yellow,
    arc=0pt,
    outer arc=0pt,
    boxrule=0pt,
    top=2pt,
    bottom=2pt,
    left=2pt,
    right=2pt,
  }
}

\newsiamthm{prop}{Proposition}
\newsiamremark{remark}{Remark}
\newsiamremark{assumption}{Assumption}

\numberwithin{equation}{section}
\numberwithin{theorem}{section}

\title{Modified electromagnetic transmission eigenvalues in inverse scattering theory\thanks{A short summary of the main results of this work have been submitted for publication in conference proceedings.}}
\author{S. Cogar\thanks{Department of Mathematics, Rutgers University, Piscataway, NJ 08854 ({\color{blue}samuel.cogar\myat rutgers.edu}).} \and P. Monk\thanks{Department of Mathematical Sciences, University of Delaware, Newark, DE 19716 ({\color{blue}monk\myat udel.edu}).\funding{This material is based upon work supported by the Army Research Office through the National Defense Science and Engineering Graduate (NDSEG) Fellowship, 32 CFR 168a, and by the Air Force Office of Scientific Research under Award No. FA9550-17-1-0147. }}}
\date{}

\begin{document}

\maketitle

\begin{abstract}
A recent problem of interest in inverse problems has been the study of eigenvalue problems arising from scattering theory and their potential use as target signatures in nondestructive testing of materials. Towards this pursuit we introduce a new eigenvalue problem related to Maxwell's equations that is generated from a comparison of measured scattering data to that of a non-standard auxiliary scattering problem. This choice of auxiliary problem permits the application of regularity results for Maxwell's equations in order to show that a related interior transmission problem possesses the Fredholm property, which is used to establish that the eigenvalues are discrete. We investigate the properties of this new class of eigenvalues and show that the eigenvalues may be determined from measured scattering data, concluding with a simple demonstration of this result.
\end{abstract}

\begin{keywords}
	inverse scattering, nondestructive testing, modified transmission eigenvalues, linear sampling method, Maxwell's equations
\end{keywords}

\begin{AMS}
	35J25, 35P05, 35P25, 35R30
\end{AMS}

\section{Introduction} \label{sec_introduction}

A recent area of research in inverse problems has been the development and study of new eigenvalue problems arising from scattering theory (cf. \cite{audibert_cakoni_haddar,audibert_chesnel_haddar,audibert_chesnel_haddar2019,cakoni_colton_meng_monk,camano_lackner_monk,cogar,cogar2019,cogar2020,cogar_colton_meng_monk,cogar_colton_monk,cogar_colton_monk2019}). In addition to generating mathematically interesting problems, the study of the resulting eigenvalues may find potential application in nondestructive testing of materials. In particular, eigenvalues could potentially be used as a target signature in order to characterize a scattering medium and indicate when it has experienced some perturbation of its material coefficients. An early example is the use of transmission eigenvalues, and we refer to \cite{cakoni_colton_haddar} for a comprehensive treatment of the subject. Transmission eigenvalues carry information about the medium of interest and noticeably shift in response to changes in the medium, but their detection from scattering data requires multifrequency data. In addition, only real transmission eigenvalues can be detected, and for an absorbing medium no real transmission eigenvalues exist. \par

In order to remedy these shortcomings, new eigenvalue problems have been generated by comparing the measured scattering data to that of an auxiliary scattering problem that depends on a parameter and is entirely artificial, i.e. it does not depend upon the physical medium of interest. If we seek values of this parameter for which the measured and auxiliary scattering data might coincide for certain types of incident fields, then we arrive at an eigenvalue problem dependent upon the material coefficients of the medium in which this parameter serves as the eigenvalue. Each choice of auxiliary scattering problem produces a new eigenvalue problem that may potentially be used to detect flaws in a medium. The first example of this approach in \cite{cakoni_colton_meng_monk} featured auxiliary data from an exterior impedance problem with parameter $\lambda$, which resulted in the well-known Stekloff eigenvalue problem. Through a series of numerical examples in two dimensions the authors demonstrated that Stekloff eigenvalues may be detected from measured scattering data and that they shift in response to changes in the refractive index of the medium. \par

The sensitivity of Stekloff eigenvalues to changes in the scattering medium was not always found to be significant, and this observation led the authors of \cite{cogar_colton_meng_monk} to choose the auxiliary data as that of scattering by an inhomogeneous medium that depends upon an additional fixed parameter $\gamma$. The resulting eigenvalue problem has a similar form to the standard transmission eigenvalue problem, but the structure and techniques used to analyze it were of a different nature. The authors demonstrated that in many cases the fixed parameter $\gamma$ may be tuned in order to increase the sensitivity of the so-called modified transmission eigenvalues to changes in the medium, often with an increase of an order of magnitude. Variations of this idea were explored in \cite{cogar,cogar2019,cogar_colton_monk}. \par

Both of the problems mentioned above related to acoustic scattering, but in \cite{camano_lackner_monk} this approach was first applied to electromagnetic scattering, which is the context of our current investigation. This first foray into generating electromagnetic eigenvalue problems again saw the choice of auxiliary data arising from an exterior impedance problem, but the resulting eigenvalue problem lacked the same solvability properties of its acoustic counterpart. In particular, the authors used a simple example to show that the eigenvalues could no longer correspond to those of a compact operator, which would prove problematic in the standard approach to establishing solvability results of the associated electromagnetic Stekloff eigenvalue problem (an issue later overcome in \cite{halla2,halla1} using $T$-coercivity). Recognizing that the auxiliary problem could be changed at will, the authors modified the boundary condition of the auxiliary problem in order to remove the degenerate branch of eigenvalues and obtained a well-behaved eigenvalue problem. Through the numerical examples in \cite{camano_lackner_monk,cogar_colton_monk2019}, it has been shown that this generalization of electromagnetic Stekloff eigenvalues is sensitive to changes in the electromagnetic properties of a medium, but like the acoustic case this shift is not always significant. \par

As in the acoustic case, this observation leads us to consider an auxiliary problem that depends on a fixed tuning parameter $\gamma$. The obvious first choice is the electromagnetic version of the problem considered in \cite{cogar_colton_meng_monk}, which represents electromagnetic scattering by a medium with constant electric permittivity $\eta$ and magnetic permeability $\gamma$. By similar reasoning to \cite{camano_lackner_monk} and \cite{cogar_colton_meng_monk} this choice results in the modified transmission eigenvalue problem in which we seek $\eta\in\mathbb{C}$ and a nonzero pair $(\mathbf{w},\mathbf{v})$ such that
\begin{subequations} \label{mp}
\begin{align}
\curl\curl\mathbf{w} - k^2\epsilon\mathbf{w} &= \mathbf{0} \text{ in } B, \label{mp1} \\
\curl\gamma^{-1}\curl\mathbf{v} - k^2\eta\mathbf{v} &= \mathbf{0} \text{ in } B, \label{mp2} \\
\un\times(\mathbf{w} - \mathbf{v}) &= \mathbf{0} \text{ on } \partial B, \label{mp3} \\
\un\times(\curl\mathbf{w} - \gamma^{-1}\curl\mathbf{v}) &= \mathbf{0} \text{ on } \partial B, \label{mp4}
\end{align}
\end{subequations}
where $\epsilon$ is the relative electric permittivity of the physical medium, $k>0$ is the wave number, and $B$ is a Lipschitz domain in $\mathbb{R}^3$ containing the support of $1-\epsilon$. We will provide more assumptions on these quantities in the next section, but for now we examine the structure of the eigenvalues of \eqref{mp} in the simple case where $B$ is the unit ball in $\mathbb{R}^3$ and $\epsilon$ is a constant in $B$. \par

In this case we may use separation of variables in order to solve \eqref{mp} in a similar manner to \cite{camano_lackner_monk}, and the result is that $\eta\neq0$ is an eigenvalue if and only if for some $n\in\mathbb{N}_0$ it is a zero of one of the determinant functions
\begin{subequations} \label{det}
\begin{align}
d_n^{(a)}(\eta) &= \left(1-\gamma^{-1}\right)j_n(k\sqrt{\epsilon})j_n(k\sqrt{\gamma\eta}) + k\sqrt{\epsilon} j_n'(k\sqrt{\epsilon})j_n(k\sqrt{\gamma\eta}) - k\sqrt{\gamma^{-1}\eta} j_n(k\sqrt{\epsilon}) j_n'(k\sqrt{\gamma\eta}), \label{det_a} \\
d_n^{(b)}(\eta) &= (\eta-\epsilon)j_n(k\sqrt{\epsilon})j_n(k\sqrt{\gamma\eta}) + k\eta\sqrt{\epsilon} j_n'(k\sqrt{\epsilon})j_n(k\sqrt{\gamma\eta}) - k\epsilon\sqrt{\gamma\eta} j_n(k\sqrt{\epsilon}) j_n'(k\sqrt{\gamma\eta}), \label{det_b}
\end{align}
\end{subequations}
where $j_n$ is the spherical Bessel function of the first kind of order $n$. Unlike in \cite{camano_lackner_monk}, we cannot simply solve for the eigenvalues in this case, as they are roots of a family of transcendental functions. Thus, we instead provide a plot of these roots in Figure \ref{fig_sov1} for $k = 1$, $\epsilon = 2$, $\gamma = 0.5$.

\begin{figure}
\begin{center}
\includegraphics[scale=0.5]{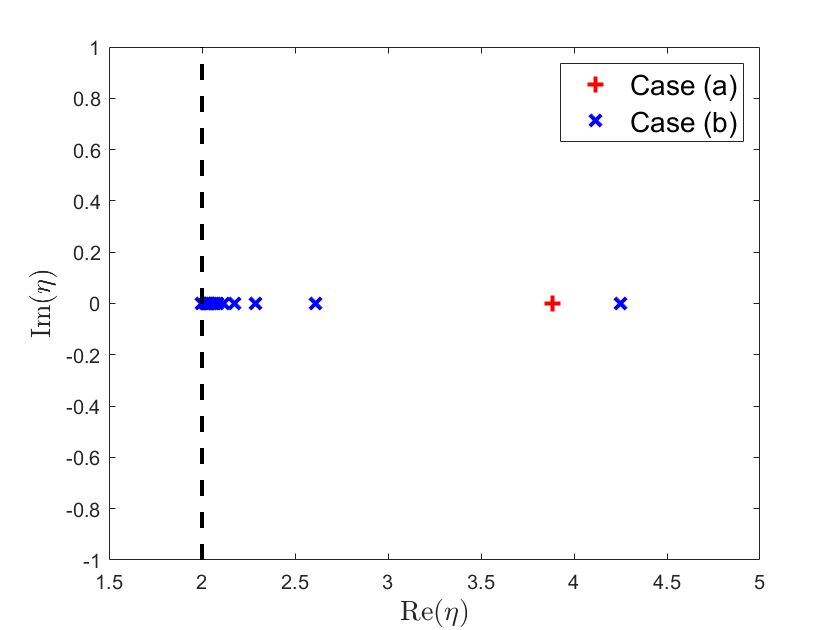}
\caption{The first few eigenvalues of \eqref{mp} computed using separation variables as the roots of the determinant functions defined in \eqref{det}. The roots of $d_n^{(a)}$ and $d_n^{(b)}$ are marked by a red $+$ symbol and a blue $\times$ symbol, respectively. The vertical dashed line marks the constant value of the permittivity $\epsilon$.}
\label{fig_sov1}
\end{center}
\end{figure}

We see from Figure \ref{fig_sov1} that the roots of the family $\{d_n^{(a)}\}$ appear to diverge towards $+\infty$, whereas the roots of $\{d_n^{(b)}\}$ do not. In \cite{chesnel} it is shown that the set of standard transmission eigenvalues for Maxwell's equations is discrete without finite accumulation point whenever $\epsilon-1$ is bounded away from zero. Performing the same calculations in the present case implies the same result for the eigenvalues of \eqref{mp} whenever $\epsilon-\eta$ is bounded away from zero, i.e. the eigenvalues are discrete without finite accumulation point in the domain $\mathbb{C}\setminus\{\epsilon\}$ in the case of constant $\epsilon$. From this result we see that the only possible finite accumulation point in this case is $\eta=2$, and we observe this accumulation point of the case (b) eigenvalues in Figure \ref{fig_sov1}. \par

We have thus encountered the same difficulty as in \cite{camano_lackner_monk}, in which the eigenvalues accumulate at both infinity and some finite point. The usual approach in studying this type of eigenvalue problem is to define a solution operator $\Psi$ whose spectrum is related to the eigenvalues of \eqref{mp} and to prove that $\Psi$ is compact. However, the spectral theorem for compact operators implies that the eigenvalues of $\Psi$ must accumulate only at zero, and hence the eigenvalues of \eqref{mp} may only accumulate at infinity. Therefore, our numerical evidence suggests that the eigenvalues of \eqref{mp} cannot be related to the spectrum of a compact operator, and we have lost one of our main analytical tools. While we note that techniques similar to those used in \cite{halla2,halla1} for the unmodified electromagnetic Stekloff eigenvalue problem might applied to analyze \eqref{mp}, it is not immediately clear how to do so. As a consequence, properties of these eigenvalues are currently unknown. \par


This observation motivates us to consider a different eigenvalue problem that will explicitly force compactness of the resulting solution operator $\Psi$, in which we seek $\eta\in\mathbb{C}$ and a nonzero triple $(\mathbf{w},\mathbf{v},p)$ such that
\begin{subequations} \label{modp}
\begin{align}
\curl\curl\mathbf{w} - k^2\epsilon\mathbf{w} &= \mathbf{0} \text{ in } B, \label{modp1} \\
\curl\gamma^{-1}\curl\mathbf{v} - k^2\eta\mathbf{v} + k^2\nabla p &= \mathbf{0} \text{ in } B, \label{modp2} \\
\div\mathbf{v} &= 0 \text{ in } B, \label{modp3} \\
\un\cdot\mathbf{v} &= 0 \text{ on } \partial B, \label{modp4} \\
\un\times(\mathbf{w} - \mathbf{v}) &= \mathbf{0} \text{ on } \partial B, \label{modp5} \\
\un\times(\curl\mathbf{w} - \gamma^{-1}\curl\mathbf{v}) &= \mathbf{0} \text{ on } \partial B. \label{modp6}
\end{align}
\end{subequations}

Whereas this goal was accomplished in \cite{camano_lackner_monk} by essentially removing the problematic branch of eigenvalues corresponding to case (b), the new problem \eqref{modp} only modifies this branch of eigenvalues. Solving the problem using the separation of variables approach in the Appendix shows that the branch corresponding to case (a) is unchanged, but the branch corresponding to case (b) is dramatically different. In particular, we no longer observe a sequence of eigenvalues converging to a finite point. In fact, the smallest eigenvalue for case (b) is approximately $\eta = 18.317$.



The outline of this paper is as follows. In Section \ref{sec_aux} we introduce the physical scattering problem of interest and the auxiliary problem that we will use in order to generate the eigenvalue problem \eqref{modp}, and we establish that the auxiliary problem is well-posed. The goal of Section \ref{sec_mitp} is to prove a solvability result for a nonhomogeneous version of \eqref{modp} that will allow us to study the properties of the eigenvalues. We begin this investigation in Section \ref{sec_props} by showing that the eigenvalues are discrete without finite accumulation point, and we establish that eigenvalues exist whenever $\epsilon$ is real-valued. In Section \ref{sec_determine} we show that the eigenvalues may be determined from measured scattering data using the linear sampling method. Section \ref{sec_numerics} is devoted to the presentation of a simple numerical example for scattering by a ball in order to illustrate the method. We conclude with a discussion of some open questions and potential avenues of research in Section \ref{sec_conclusion}, followed by a short appendix that provides some details regarding the separation of variables procedure mentioned above.

\section{The physical and auxiliary scattering problems} \label{sec_aux}

Given an incident electric field $\mathbf{E}^i$ which satisfies the free-space Maxwell's equations in $\mathbb{R}^3$, we seek a scattered field $\mathbf{E}^s\in\Hcurlloc{\mathbb{R}^3\setminus\overline{D}}$ and a total field $\mathbf{E}\in\Hcurl{D}$ which satisfy the standard Maxwell system
\begin{subequations} \label{sc}
\begin{align}
\curl\curl \mathbf{E}^s - k^2 \mathbf{E}^s &= \mathbf{0} \text{ in } \mathbb{R}^3\setminus\overline{D}, \label{sc1} \\
\curl\curl \mathbf{E} - k^2\epsilon \mathbf{E} &= \mathbf{0} \text{ in } D, \label{sc2} \\
\un\times\mathbf{E} - \un\times\mathbf{E}^s &= \un\times \mathbf{E}^i \text{ on } \partial D, \label{sc3} \\
\un\times\curl \mathbf{E} - \un\times\curl \mathbf{E}^s &= \un\times\curl \mathbf{E}^i \text{ on } \partial D, \label{sc4} \\
\mathclap{\lim_{r\to\infty} \left(\curl \mathbf{E}^s\times\mathbf{x} - ikr\mathbf{E}^s\right) = 0,} \label{sc5}
\end{align}
\end{subequations}
where $\epsilon$ is the relative electric permittivity of the medium, $k>0$ is the wave number, $\overline{D}$ is the support of the contrast $\epsilon - 1$, and $\un$ is the outward unit normal vector of the boundary $\partial D$. We assume that $\epsilon = 1$ outside of a sufficiently large ball centered at the origin, which implies that $D$ is bounded. We also assume that $D$ is a Lipschitz domain with connected complement and that $\epsilon$ satisfies $\Re(\epsilon)\ge\epsilon_*>0$ and $\Im(\epsilon)\ge0$ a.e. in $D$. In order to permit the application of the unique continuation principle, we assume that $\epsilon|_D$ lies in the space
\begin{equation*}
W_{\Sigma}^{1,\infty}(D) := \{\mu\in L^\infty(D) \mid \nabla(\mu|_{\Omega_i})\in\mathbf{L}^\infty(\Omega_i), \; i = 1,2,\dots,M\},
\end{equation*}
where $\{\Omega_i\}_{i=1}^M$ is a partition of $D$. We refer to \eqref{sc} as the physical scattering problem, and under the assumptions given above this problem is well-posed for any incident field $\mathbf{E}^i$ (cf. \cite{colton_kress}). \par

We now introduce an auxiliary scattering problem that will allow us to generate an eigenvalue problem that depends on the permittivity $\epsilon$, but we remark that the auxiliary problem itself is independent of this parameter. We choose a bounded Lipschitz domain $B\subset\mathbb{R}^3$ that contains $D$ (e.g. a ball or $B=D$), and we seek $\mathbf{E}_0^s\in \Hcurlloc{\mathbb{R}^3\setminus\overline{B}}$, $\mathbf{E}_0\in \Hcurl{B}$, and $P\in H_*^1(B) := H^1(B)/\mathbb{C}$ which satisfy
\begin{subequations} \label{aux}
\begin{align}
\curl\curl \mathbf{E}_0^s - k^2 \mathbf{E}_0^s &= \mathbf{0} \text{ in } \mathbb{R}^3\setminus\overline{B}, \label{aux1} \\
\curl\gamma^{-1}\curl \mathbf{E}_0 - k^2\eta \mathbf{E}_0 + k^2\nabla P &= \mathbf{0} \text{ in } B, \label{aux2} \\
\div \mathbf{E}_0 &= 0 \text{ in } B, \label{aux3} \\
\un\cdot \mathbf{E}_0 &= 0 \text{ on } \partial B, \label{aux4} \\
\un\times\mathbf{E}_0 - \un\times\mathbf{E}_0^s &= \un\times \mathbf{E}^i \text{ on } \partial B, \label{aux5} \\
\un\times\gamma^{-1}\curl \mathbf{E}_0 - \un\times\curl \mathbf{E}_0^s &= \un\times\curl \mathbf{E}^i \text{ on } \partial B, \label{aux6} \\
\mathclap{\lim_{r\to\infty} \left(\curl \mathbf{E}_0^s\times\mathbf{x} - ikr\mathbf{E}_0^s\right) = 0,} \label{aux7}
\end{align}
\end{subequations}
where $\gamma>0$ is a fixed constant and $\eta\in\mathbb{C}$ is the parameter of interest that will later serve as an eigenvalue. We would like to establish solvability of this nonstandard problem, and we begin by showing uniqueness of solutions whenever $\Im(\eta)\ge0$.

\begin{theorem} \label{theorem_uniqueness}

If $\Im(\eta)\ge0$, then there exists at most one solution of \eqref{aux} for a given incident field $\mathbf{E}^i$.

\end{theorem}

\begin{proof}

By linearity it suffices to show that the only solution of \eqref{aux} corresponding to $\mathbf{E}^i = \mathbf{0}$ is $(\mathbf{E}_0^s,\mathbf{E}_0,P) = (\mathbf{0},\mathbf{0},0)$. Indeed, we suppose that $(\mathbf{E}_0^s,\mathbf{E}_0,P)$ is a solution for $\mathbf{E}^i = \mathbf{0}$, and we extend $\mathbf{E}_0$ as $\mathbf{E}_0 := \mathbf{E}_0^s$ in $\mathbb{R}^3\setminus\overline{B}$, which lies in $\Hcurlloc{\mathbb{R}^3}$ due to the boundary condition \eqref{aux5}. Setting $\mathbf{H}_0 := \frac{1}{ik}\curl\mathbf{E}_0$ and integrating by parts against $\overline{\mathbf{E}_0}$ in \eqref{aux1}, we see that
\begin{equation*}
\int_{B_R\setminus\overline{B}} \left( \abs{\curl\mathbf{E}_0}^s - k^2\abs{\mathbf{E}_0}^2 \right) dx - ik\int_{\partial B_R} \un\times\overline{\mathbf{E}_0}\cdot\mathbf{H}_0 ds - \int_{\partial B} (\un\times\curl\mathbf{E}_0)\cdot\overline{\mathbf{E}_{0,T}} ds = 0,
\end{equation*}
where $B_R$ is a ball centered at the origin chosen such that $\overline{B}\subset B_R$. If we integrate by parts in \eqref{aux2} in a similar manner and apply the transmission conditions \eqref{aux5}--\eqref{aux6}, then we obtain
\begin{align*}
\int_{B_R\setminus\overline{B}} \left( \abs{\curl\mathbf{E}_0}^2 - k^2\abs{\mathbf{E}_0}^2 \right) dx - ik\int_{\partial B_R} &\un\times\overline{\mathbf{E}_0}\cdot\mathbf{H}_0 ds \\
+ \int_B \left(\gamma^{-1}\abs{\curl\mathbf{E}_0}^2 - k^2\eta\abs{\mathbf{E}_0}^2 \right) &dx + k^2\int_B \nabla P\cdot\overline{\mathbf{E}_0} dx = 0.
\end{align*}
The vanishing divergence and normal component of $\mathbf{E}_0$ required by \eqref{aux3}--\eqref{aux4} imply that the last integral on the left-hand side vanishes,
and by taking the imaginary part of both sides it follows that
\begin{equation*}
\Re\int_{\partial B_R} \un\times\overline{\mathbf{E}_0}\cdot\mathbf{H}_0 ds = -k\Im(\eta)\int_B \abs{\mathbf{E}_0}^2 dx \le 0.
\end{equation*}
By Rellich's lemma (cf. \cite{colton_kress}) we see that $\mathbf{E}_0 = \mathbf{0}$ in $\mathbb{R}^3\setminus\overline{B}_R$, and the unique continuation principle for Maxwell's equations implies further that $\mathbf{E}_0 = \mathbf{0}$ in $\mathbb{R}^3\setminus\overline{B}$. In particular, we observe from the transmission conditions \eqref{aux5}--\eqref{aux6} that $\un\times\mathbf{E}_0 = \un\times\gamma^{-1}\curl\mathbf{E}_0 = \mathbf{0} \text{ on } \partial B$,
and as a consequence we may integrate by parts in \eqref{aux2} against $\nabla\overline{P}$ to obtain
\begin{equation*}
k^2\int_B \abs{\nabla P}^2 dx - k^2\eta\int_B \mathbf{E}_0\cdot\nabla\overline{P} dx = 0.
\end{equation*}
From \eqref{aux3}--\eqref{aux4} we see that the second integral vanishes, which implies that $\nabla P = \mathbf{0}$ and hence $P = 0$ in $B$ since $P\in H_*^1(B)$. Finally, we see that $\mathbf{E}_0$ satisfies
\begin{equation*}
\curl{\tilde\gamma}^{-1}\curl\mathbf{E}_0 - k^2\tilde{\eta}\mathbf{E}_0 = \mathbf{0} \text{ in } \mathbb{R}^3,
\end{equation*}
where for any constant $\alpha$ we define $\tilde{\alpha} := \alpha$ in $B$ and $\tilde{\alpha} := 1$ elsewhere. Since $\mathbf{E}_0$ is identically zero in $\mathbb{R}^3\setminus\overline{B}$, the unique continuation principle implies that $\mathbf{E}_0 = \mathbf{0}$ in $B$ as well, and we conclude that $(\mathbf{E}_0^s,\mathbf{E}_0,P) = (\mathbf{0},\mathbf{0},0)$. \proofend
\end{proof}

We now aim to show that \eqref{aux} is well-posed whenever $\Im(\eta)\ge0$, and in particular we show that this problem is of Fredholm type, i.e. existence of solutions follows from uniqueness. In the following remark we introduce two modifications of the problem that will allow us to derive an equivalent variational formulation of \eqref{aux}. 

\begin{remark} \label{remark_mod1}
\textnormal{
First, we choose $\boldsymbol{\varphi}\in \Hcurlloc{\mathbb{R}^3\setminus\overline{B}}$ to be the unique radiating solution of
\begin{align*}
\curl\curl\boldsymbol{\varphi} - k^2\boldsymbol{\varphi} &= \mathbf{0} \text{ in } \mathbb{R}^3\setminus\overline{B}, \\
\un\times\boldsymbol{\varphi} &= \un\times\mathbf{E}^i \text{ on } \partial B.
\end{align*}
By the well-posedness of this standard problem there exists a constant $C_K$ independent of $\mathbf{E}^i$ such that
\begin{equation*}
\norm{\boldsymbol{\varphi}}_{\Hcurl{K})} \le C_K\norm{\un\times\mathbf{E}^i}_{\mathbf{H}^{-1/2}(\Div,\partial B)},
\end{equation*}
where $K$ is any bounded subset of $\mathbb{R}^3\setminus\overline{B}$. We write $\mathbf{u} := \mathbf{E}_0^s + \boldsymbol{\varphi}$, and we observe from the boundary conditions \eqref{aux5}--\eqref{aux6} that
\begin{align*}
\un\times\mathbf{E}_0 - \un\times\mathbf{u} &= \mathbf{0} \text{ on } \partial B, \\
\un\times\gamma^{-1}\curl \mathbf{E}_0 - \un\times\curl \mathbf{u} &= \un\times\curl \mathbf{E}^i - \un\times\curl\boldsymbol{\varphi} \text{ on } \partial B.
\end{align*}
Second, we choose $\zeta\in H_*^1(B)$ to be the unique solution of
\begin{align*}
\Delta\zeta &= 0 \text{ in } B, \\
\normal{\zeta} &= k^{-2}\nabla_{\partial B}\cdot\left( \un\times\curl \mathbf{E}^i - \un\times\curl\boldsymbol{\varphi} \right) \text{ on } \partial B,
\end{align*}
where $\nabla_{\partial B}\cdot$ denotes the surface divergence on $\partial B$, and we remark that in a similar manner there exists a constant $C>0$ independent of $\mathbf{E}^i$ and $\boldsymbol{\varphi}$ such that
\begin{equation*}
\norm{\nabla\zeta}_B \le C\norm{\nabla_{\partial B}\cdot\left( \un\times\curl \mathbf{E}^i - \un\times\curl\boldsymbol{\varphi} \right)}_{H^{-1/2}(\partial B)}.
\end{equation*}
We write $p := P-\zeta$ in $B$, and we see from \eqref{aux2} that
\begin{equation*}
\curl\gamma^{-1}\curl\mathbf{E}_0 - k^2\eta\mathbf{E}_0 + k^2\nabla p = -k^2\nabla\zeta \text{ in } B.
\end{equation*}
The reason for this modification is to guarantee that certain relationships between the solution fields are homogeneous in our upcoming analysis.}
\end{remark}

In addition to the modifications from Remark \ref{remark_mod1}, we formulate \eqref{aux} on a bounded domain $B_R$ using the electric-to-magnetic Calder{\'o}n operator $G_e:\mathbf{H}^{-1/2}(\Div,\partial B_R)\to\mathbf{H}^{-1/2}(\Div,\partial B_R)$, which maps the tangential component of the electric field on $\partial B_R$ to the tangential component of the magnetic field on $\partial B_R$ that arises from the unique radiating solution of the free-space Maxwell's equations in the exterior domain $\mathbb{R}^3\setminus\overline{B_R}$. We refer to \cite{monk} for details on this operator. This operator serves to replace the equation and radiation condition for the scattered electric field in the exterior domain $\mathbb{R}^3\setminus\overline{B_R}$ with a boundary condition on $\partial B_R$. \par

If we write $\mathbf{v} := \mathbf{E}_0$ for convenience, then an equivalent formulation of \eqref{aux} is to seek $\mathbf{u}\in\Hcurl{B_R\setminus\overline{B}}$, $\mathbf{v}\in\Hcurl{B}$, and $p\in H_*^1(B)$ which satisfy
\begin{subequations} \label{auxb}
\begin{align}
\curl\curl \mathbf{u} - k^2 \mathbf{u} &= \mathbf{0} \text{ in } B_R\setminus\overline{B}, \label{auxb1} \\
\curl\gamma^{-1}\curl \mathbf{v} - k^2\eta \mathbf{v} + k^2\nabla p &= -k^2\nabla\zeta \text{ in } B, \label{auxb2} \\
\div \mathbf{v} &= 0 \text{ in } B, \label{auxb3} \\
\un\cdot \mathbf{v} &= 0 \text{ on } \partial B, \label{auxb4} \\
\un\times\mathbf{v} - \un\times\mathbf{u} &= \mathbf{0} \text{ on } \partial B, \label{auxb5} \\
\un\times\gamma^{-1}\curl \mathbf{v} - \un\times\curl \mathbf{u} &= \mathbf{h} \text{ on } \partial B, \label{auxb6} \\
\un\times\curl\mathbf{u} &= ik G_e(\un\times\mathbf{u}) \text{ on } \partial B_R, \label{auxb7}
\end{align}
\end{subequations}
where $\mathbf{h} := \un\times\curl\mathbf{E}^i - \un\times\curl\boldsymbol{\varphi}\in \mathbf{H}^{-1/2}(\Div,\partial B)$. In order to study an equivalent variational formulation of \eqref{auxb} we introduce the space
\begin{equation*}
\boldsymbol{\mathcal{X}} := \left\{(\mathbf{u},\mathbf{v},p)\in\Hcurl{B_R\setminus\overline{B}}\times\Hcurl{B}\times H_*^1(B) \;\middle|\; \un\times\mathbf{u}-\un\times\mathbf{v} = \mathbf{0} \text{ on } \partial B\right\},
\end{equation*}
equipped with the usual inner product $(\cdot,\cdot)_{\boldsymbol{\mathcal{X}}}$ and induced norm $\norm{\cdot}_{\boldsymbol{\mathcal{X}}}$ inherited from the component spaces. If $(\mathbf{u},\mathbf{v},p)$ satisfies \eqref{auxb} and we integrate by parts in \eqref{auxb1}--\eqref{auxb3} against the test function components $(\mathbf{u}',\mathbf{v}',p')\in \boldsymbol{\mathcal{X}}$, then we obtain
\begin{align*}
(\curl\mathbf{u},\curl\mathbf{u}')_{B_R\setminus\overline{B}} - k^2(\mathbf{u},\mathbf{u}')_{B_R\setminus\overline{B}} + \inner{\un\times\curl\mathbf{u}}{\mathbf{u}_T'}_{\partial B_R} - \inner{\un\times\curl\mathbf{u}}{\mathbf{u}_T'}_{\partial B} &= 0, \\
(\curl\mathbf{v},\curl\mathbf{v}')_B - k^2\eta(\mathbf{v},\mathbf{v}')_B + \inner{\un\times\gamma^{-1}\curl\mathbf{v}}{\mathbf{v}_T'}_{\partial B} + k^2(\nabla p,\mathbf{v}')_B &= -k^2(\nabla\zeta,\mathbf{v}')_B, \\
(\mathbf{v},\nabla p')_B &= 0,
\end{align*}
where for a Lipschitz domain $\mathcal{O}\subseteq\mathbb{R}^3$ with boundary $\partial\mathcal{O}$ we denote by $(\cdot,\cdot)_\mathcal{O}$ the inner product on $\mathbf{L}^2(\mathcal{O})$ and by $\inner{\cdot}{\cdot}_{\partial\mathcal{O}}$ the duality pairing of $\mathbf{H}^{-1/2}(\Div,\partial\mathcal{O})$ and $\mathbf{H}^{-1/2}(\Curl,\partial\mathcal{O})$ (with the second argument conjugated). In some instances we will also use $\inner{\cdot}{\cdot}_{\partial\mathcal{O}}$ to denote the duality pairing of $H^{-1/2}(\partial\mathcal{O})$ and $H^{1/2}(\partial\mathcal{O})$ for scalar functions, and we will sometimes use the shorthand $\norm{\cdot}_B$ to represent the norms $\norm{\cdot}_{L^2(B)}$ and $\norm{\cdot}_{\mathbf{L}^2(B)}$. The combination of these equations along with the boundary conditions yields a variational problem in which we seek $(\mathbf{u},\mathbf{v},p)\in \boldsymbol{\mathcal{X}}$ satisfying
\begin{equation}
a((\mathbf{u},\mathbf{v},p),(\mathbf{u}',\mathbf{v}',p')) = \ell(\mathbf{u}',\mathbf{v}',p') \quad\forall(\mathbf{u}',\mathbf{v}',p')\in \boldsymbol{\mathcal{X}}, \label{varprob_aux}
\end{equation}
where the bounded sesquilinear form $a(\cdot,\cdot)$ is defined by
\begin{align*}
a((\mathbf{u},\mathbf{v},p),(\mathbf{u}',\mathbf{v}',p')) &:= (\curl\mathbf{u},\curl\mathbf{u}')_{B_R\setminus\overline{B}} + \gamma^{-1}(\curl\mathbf{v},\curl\mathbf{v}')_B - k^2(\mathbf{u},\mathbf{u}')_{B_R\setminus\overline{B}} \\
&\quad\quad\quad- k^2\eta(\mathbf{v},\mathbf{v}')_B + k^2(\nabla p,\mathbf{v}')_B + k^2(\mathbf{v},\nabla p')_B \\
&\quad\quad\quad+ ik\inner{G_e(\un\times\mathbf{u})}{\mathbf{u}_T'}_{\partial B_R} \quad\forall(\mathbf{u},\mathbf{v},p),(\mathbf{u}',\mathbf{v}',p')\in \boldsymbol{\mathcal{X}},
\end{align*}
and the bounded antilinear functional $\ell$ is defined by
\begin{equation*}
\ell(\mathbf{u}',\mathbf{v}',p') := -k^2(\nabla\zeta,\mathbf{v}')_B - \inner{\mathbf{h}}{\mathbf{v}_T'}_{\partial B} \quad\forall(\mathbf{u}',\mathbf{v}',p')\in \boldsymbol{\mathcal{X}}.
\end{equation*}
We now investigate the properties of solutions of \eqref{varprob_aux}, and we begin by introducing the space
\begin{equation*}
S := \left\{ (\varphi,\psi,q)\in H^1(B_R\setminus\overline{B})\times H^1(B)\times H_*^1(B) \;\middle|\; \begin{array}{c} \varphi = \psi \text{ on } \partial B \\ \inner{\varphi}{1}_{\partial B} = \inner{\psi}{1}_{\partial B} = 0 \end{array} \right\}.
\end{equation*}
For any $(\varphi,\psi,q)\in S$ it follows that $\un\times\nabla\varphi - \un\times\nabla\psi = 0$ on $\partial B$ since $\varphi = \psi$ on $\partial B$, and as a result we have $(\nabla\varphi,\nabla\psi,q)\in \boldsymbol{\mathcal{X}}$. With $(\mathbf{u}',\mathbf{v}',p) = (\nabla\varphi,\nabla\psi,q)$ we see that any solution $(\mathbf{u},\mathbf{v},p)$ of \eqref{varprob_aux} must satisfy
\begin{align}
\begin{split}
k^2(\mathbf{u},&\nabla\varphi)_{B_R\setminus\overline{B}} + k^2\eta(\mathbf{v},\nabla\psi)_B - k^2(\nabla p,\nabla\psi)_B - k^2(\mathbf{v},\nabla q)_B \\
&\quad- ik\inner{G_e(\un\times\mathbf{u})}{\nabla_{\partial B_R}\varphi}_{\partial B_R} = k^2(\nabla\zeta,\nabla\psi)_B + \inner{\mathbf{h}}{\nabla_{\partial B}\psi}_{\partial B}. \label{varS}
\end{split}
\end{align}
We first observe that by choosing $\varphi = 0$ and $\psi = 0$ we have
\begin{equation*}
(\mathbf{v},\nabla q)_B = 0 \quad\forall q\in H_*^1(B),
\end{equation*}
which reflects the conditions \eqref{auxb3}--\eqref{auxb4} in the reformulated auxiliary problem. By choosing $\varphi\in H_0^1(B_R\setminus\overline{B})$, $\psi = 0$, and $q=0$ we have
\begin{equation*}
(\mathbf{u},\nabla\varphi)_{B_R\setminus\overline{B}} = 0 \quad\forall\varphi\in H_0^1(B_R\setminus\overline{B}),
\end{equation*}
and by instead choosing $\varphi = 0$, $\psi\in H_0^1(B)$, and $q=0$ we observe that
\begin{equation*}
(\nabla p,\nabla\psi)_B = -(\nabla\zeta,\nabla\psi)_B = 0 \quad\forall\psi\in H_0^1(B)
\end{equation*}
since $\Delta\zeta = 0$ in $B$ by construction. Thus, it follows that
\begin{equation*}
\div\mathbf{u} = 0 \text{ in } B_R\setminus\overline{B} \text{ and } \Delta p = 0 \text{ in } B,
\end{equation*}
and applying the divergence theorem yields
\begin{align*}
(\mathbf{u},\nabla\varphi)_{B_R\setminus\overline{B}} &= \inner{\un\cdot\mathbf{u}}{\varphi}_{\partial B_R} - \inner{\un\cdot\mathbf{u}}{\varphi}_{\partial B}, \\
(\nabla p,\nabla\psi)_B &= \inner{\normal{p}}{\psi}_{\partial B}, \\
(\nabla\zeta,\nabla\psi)_B &= k^{-2}\inner{\nabla_{\partial B}\cdot\mathbf{h}}{\psi}_{\partial B},
\end{align*}
where the last equation follows from the definition of $\mathbf{h}$ and the construction of $\zeta$. If we substitute these equations into \eqref{varS} and use the definition of the surface divergence (cf. \cite{monk}), then we see that
\begin{equation*}
\inner{\un\cdot\mathbf{u} - \frac{1}{ik}\nabla_{\partial B_R}\cdot G_e(\un\times\mathbf{u})}{\varphi}_{\partial B_R} - \inner{\un\cdot\mathbf{u} + \normal{p}}{\psi}_{\partial B} = 0.
\end{equation*}
Choosing $\psi = 0$ in $B$ yields
\begin{equation}
\un\cdot\mathbf{u} - \frac{1}{ik}\nabla_{\partial B_R}\cdot G_e(\un\times\mathbf{u}) = 0 \text{ on } \partial B_R, \label{eq_X01}
\end{equation}
and choosing $\varphi$ such that $\varphi = 0$ near $\partial B_R$ yields
\begin{equation}
\un\cdot\mathbf{u} + \normal{p} = 0 \text{ on } \partial B. \label{eq_X02}
\end{equation}
We introduced $\zeta$ in Remark \ref{remark_mod1} in order to ensure that \eqref{eq_X01} and \eqref{eq_X02} are homogeneous and hence may be used to define a subspace of $\boldsymbol{\mathcal{X}}$. In particular, we define the solution space
\begin{equation*}
\boldsymbol{\mathcal{X}}_0 := \left\{ (\mathbf{u},\mathbf{v},p)\in \boldsymbol{\mathcal{X}} \;\middle|\; \begin{array}{c} \div\mathbf{u} = 0 \text{ in } B_R\setminus\overline{B}, \,\div\mathbf{v} = 0 \text{ in } B, \,\Delta p = 0 \text{ in } B, \\
\un\cdot\mathbf{u} - \frac{1}{ik}\nabla_{\partial B_R}\cdot G_e(\un\times\mathbf{u}) = 0 \text{ on } \partial B_R, \\
\un\cdot\mathbf{u}+\normal{p} = 0 \text{ on } \partial B, \,\un\cdot\mathbf{v} = 0 \text{ on } \partial B \end{array} \right\},
\end{equation*}
equipped with the same inner product and norm as $\boldsymbol{\mathcal{X}}$, and we observe that \eqref{varprob_aux} may be equivalently posed with $\boldsymbol{\mathcal{X}}_0$ in place of $\boldsymbol{\mathcal{X}}$. A necessary ingredient to establishing the Fredholm property of \eqref{varprob_aux} is a compactness result for the space $\boldsymbol{\mathcal{X}}_0$, and our main tool is the following theorem (cf. \cite[Theorem 2]{costabel} and \cite[Theorem 3.47]{monk}).

\begin{theorem} \label{theorem_costabel}

Let $\Omega$ be a bounded Lipschitz domain in $\mathbb{R}^3$, and let $\mathbf{u}\in\Hcurl{\Omega}\cap\Hdiv{\Omega}$. Then $\un\times\mathbf{u}\in\mathbf{L}_t^2(\partial \Omega)$ if and only if $\un\cdot\mathbf{u}\in L^2(\partial \Omega)$, and in either case we have $\mathbf{u}\in\mathbf{H}^{1/2}(\Omega)$ and the estimates
\begin{subequations}
\begin{align}
\norm{\un\times\mathbf{u}}_{\mathbf{L}_t^2(\partial\Omega)} &\le C\Bigr( \norm{\mathbf{u}}_{\mathbf{L}^2(\Omega)} + \norm{\curl\mathbf{u}}_{\mathbf{L}^2(\Omega)} + \norm{\div\mathbf{u}}_{L^2(\Omega)} + \norm{\un\cdot\mathbf{u}}_{L^2(\partial\Omega)} \Bigr), \label{costabel1} \\
\norm{\un\cdot\mathbf{u}}_{L^2(\partial\Omega)} &\le C\Bigr( \norm{\mathbf{u}}_{\mathbf{L}^2(\Omega)} + \norm{\curl\mathbf{u}}_{\mathbf{L}^2(\Omega)} + \norm{\div\mathbf{u}}_{L^2(\Omega)} + \norm{\un\times\mathbf{u}}_{\mathbf{L}_t^2(\partial\Omega)} \Bigr), \label{costabel2} \\
\norm{\mathbf{u}}_{\mathbf{H}^{1/2}(\Omega)} &\le C\Bigr( \norm{\mathbf{u}}_{\mathbf{L}^2(\Omega)} + \norm{\curl\mathbf{u}}_{\mathbf{L}^2(\Omega)} + \norm{\div\mathbf{u}}_{L^2(\Omega)} + \norm{\un\times\mathbf{u}}_{\mathbf{L}_t^2(\partial\Omega)} \Bigr). \label{costabel3}
\end{align}
\end{subequations}

\end{theorem}

\begin{theorem} \label{theorem_X0}

The space $\boldsymbol{\mathcal{X}}_0$ is compactly embedded into $\boldsymbol{\mathcal{L}} := \mathbf{L}^2(B_R\setminus\overline{B})\times\mathbf{L}^2(B)\times H_*^1(B)$.

\end{theorem}

\begin{proof}

Let $\{(\mathbf{u}_j,\mathbf{v}_j,p_j)\}_{j\in\mathbb{N}}$ be a bounded sequence in $\boldsymbol{\mathcal{X}}_0$. We see in particular that $\{\mathbf{v}_j\}$ is a bounded sequence in $\Hcurl{B}$ with $\div\mathbf{v}_j = 0\in L^2(B)$ and $\un\cdot\mathbf{v}_j = 0 \in L^2(\partial B)$. Theorem \ref{theorem_costabel} then implies that $\{\mathbf{v}_j\}$ is a bounded sequence in $\mathbf{H}^{1/2}(B)$, and from \eqref{costabel1} it follows that $\{\un\times\mathbf{v}_j\}$ is bounded in $\mathbf{L}_t^2(\partial B)$. The compact embedding of $\mathbf{H}^{1/2}(B)$ into $\mathbf{L}^2(B)$ implies that we may extract a subsequence of $\{\mathbf{v}_j\}$ that converges in $\mathbf{L}^2(B)$, and we pass to the corresponding subsequence of $\{(\mathbf{u}_j,\mathbf{v}_j,p_j)\}_{j\in\mathbb{N}}$. We now turn our attention to the sequence $\{\mathbf{u}_j\}$, and we begin by choosing $\tilde{\mathbf{u}}_j\in\Hcurlloc{\mathbb{R}^3\setminus\overline{B_R}}$ to be the unique radiating solution of
\begin{align*}
\curl\curl\tilde{\mathbf{u}}_j - k^2\tilde{\mathbf{u}}_j &= \mathbf{0} \text{ in } \mathbb{R}^3\setminus\overline{B_R}, \\
\un\times\tilde{\mathbf{u}}_j &= \un\times\mathbf{u}_j \text{ on } \partial B_R.
\end{align*}
We observe that the extension
\begin{equation*}
\mathbf{u}_j^e := \left\{\begin{array}{ll} \mathbf{u}_j \text{ in } B_R\setminus\overline{B}, \\
                                           \tilde{\mathbf{u}}_j \text{ in } \mathbb{R}^3\setminus\overline{B_R}, \end{array} \right.
\end{equation*}
lies in $\Hcurlloc{\mathbb{R}^3\setminus\overline{B}}$ since the tangential component is continuous across $\partial B_R$. We may also apply the relation (cf. \cite[Remark 3.32]{monk})
\begin{equation*}
\nabla_{\partial B_R}\cdot(\un\times\boldsymbol{\xi}) = -\un\cdot(\curl\boldsymbol{\xi})|_{\partial B_R} \quad\forall\boldsymbol{\xi}\in\Hcurlloc{\mathbb{R}^3\setminus\overline{B_R}}
\end{equation*}
along with the definition of $\tilde{\mathbf{u}}_j$ in order to obtain
\begin{equation*}
\un\cdot\mathbf{u}_j = \frac{1}{ik}\nabla_{\partial B_R}\cdot G_e(\un\times\mathbf{u}_j) = \frac{1}{ik}\nabla_{\partial B_R}\cdot \left( \frac{1}{ik}\un\times\curl\tilde{\mathbf{u}}_j \right) = \frac{1}{k^2}\un\cdot\curl\curl\tilde{\mathbf{u}}_j = \un\cdot\tilde{\mathbf{u}}_j.
\end{equation*}
It follows that $\mathbf{u}_j^e\in\Hdiv{\mathbb{R}^3\setminus\overline{B}}$ with $\div\mathbf{u}_j^e = 0$ in $\mathbb{R}^3\setminus\overline{B}$. We consider a smooth cutoff function $\chi\in C_0^\infty(\mathbb{R}^3)$ such that $\chi = 1$ in $\overline{B_R}$, and we let $B_{\tilde{R}}$, $\tilde{R}>R$, be a ball centered at the origin that strictly contains the support of $\chi$. We see that each term $\chi\mathbf{u}_j^e$ satisfies
\begin{align*}
\un\times(\chi\mathbf{u}_j^e) &= \mathbf{0} \text{ on } \partial B_{\tilde{R}}, \\
\un\times(\chi\mathbf{u}_j^e) &= \un\times\mathbf{v}_j \text{ on } \partial B,
\end{align*}
which implies that the sequence $\{(\un\times(\chi\mathbf{u}_j^e))|_{\partial(B_{\tilde{R}}\setminus\overline{B})}\}$ is bounded in $\mathbf{L}_t^2(\partial(B_{\tilde{R}}\setminus\overline{B}))$. Another application of Theorem \ref{theorem_costabel} implies that $\{\mathbf{u}_j\}$ is bounded in $\mathbf{H}^{1/2}(B_R\setminus\overline{B})$ and allows us to extract a convergent subsequence of $\{\mathbf{u}_j\}$ in $\mathbf{L}^2(B_R\setminus\overline{B})$. We again pass to the corresponding subsequence of $\{(\mathbf{u}_j,\mathbf{v}_j,p_j)\}_{j\in\mathbb{N}}$. The estimate \eqref{costabel2} also implies that the sequence $\{\un\cdot\mathbf{u}_j\}$ is bounded in $L^2(\partial B)$, and from the definition of $\boldsymbol{\mathcal{X}}_0$ we see that $\left\{\un\cdot\nabla p_j\right\}$ is bounded in $L^2(\partial B)$ as well. Since $\curl\nabla p_j = 0$ and $\div\nabla p = \Delta p = 0$ in $B$, a final application of Theorem \ref{theorem_costabel} implies that $\{\nabla p_j\}$ is a bounded sequence in $\mathbf{H}^{1/2}(B)$, and we again extract a subsequence convergent in $\mathbf{L}^2(B)$. A simple argument shows that the limit of this sequence may be written as the gradient of a scalar potential, which implies that the corresponding subsequence of $\{p_j\}$ converges in $H_*^1(B)$. We conclude that there exists a subsequence of $\{(\mathbf{u}_j,\mathbf{v}_j,p_j)\}_{j\in\mathbb{N}}$ which converges in $\boldsymbol{\mathcal{L}}$. \proofend

\end{proof}

\begin{remark} \label{remark_Ge}

\textnormal{In order to decompose the sesquilinear form $a(\cdot,\cdot)$ into coercive and compact parts, we first require such a decomposition of the Calderon operator $G_e$. From \cite[Section 10.3.2]{monk} there exist operators $G_e^{(1)}, G_e^{(2)}:\mathbf{H}^{-1/2}(\Div,\partial B_R)\to\mathbf{H}^{-1/2}(\Div,\partial B_R)$ which satisfy}
\begin{enumerate}[label=(\roman*)]

\item \textnormal{$G_e = G_e^{(1)} + G_e^{(2)}$;}

\item \textnormal{$G_e^{(1)}\circ\gamma_t\circ P_1:\boldsymbol{\mathcal{X}}_0\to\mathbf{H}^{-1/2}(\Div,\partial B_R)$ is compact, where $\gamma_t:\Hcurl{B_R\setminus\overline{B}}\to\mathbf{H}^{-1/2}(\Div,\partial B_R)$ is the tangential trace operator $\mathbf{u}\mapsto\un\times\mathbf{u}|_{\partial B_R}$ and $P_1:\boldsymbol{\mathcal{X}}_0\to\Hcurl{B_R\setminus\overline{B}}$ is the projection operator $(\mathbf{u},\mathbf{v},p)\mapsto\mathbf{u}$;}

\item \textnormal{$ikG_e^{(2)}$ is nonnegative.}

\end{enumerate}

\end{remark}

We now define the operators $\mathbb{A},\mathbb{B}:\boldsymbol{\mathcal{X}}_0\to \boldsymbol{\mathcal{X}}_0$ by means of the Riesz representation theorem such that
\begin{align*}
(\mathbb{A}(\mathbf{u},\mathbf{v},p),(\mathbf{u}',\mathbf{v}',p'))_{\boldsymbol{\mathcal{X}}_0} &= (\curl\mathbf{u},\curl\mathbf{u}')_{B_R\setminus\overline{B}} + \gamma^{-1}(\curl\mathbf{v},\curl\mathbf{v}')_B \\
&\quad\quad\quad+ k^2(\mathbf{u},\mathbf{u}')_{B_R\setminus\overline{B}} + k^2(\mathbf{v},\mathbf{v}')_B + (\nabla p,\nabla p')_B \\
&\quad\quad\quad+ ik\inner{G_e^{(2)}(\un\times\mathbf{u})}{\mathbf{u}_T'}_{\partial B_R}, \\
(\mathbb{B}(\mathbf{u},\mathbf{v},p),(\mathbf{u}',\mathbf{v}',p'))_{\boldsymbol{\mathcal{X}}_0} &= -2k^2(\mathbf{u},\mathbf{u}')_{B_R\setminus\overline{B}} - k^2(1+\eta)(\mathbf{v},\mathbf{v}')_B - (\nabla p,\nabla p')_B \\
&\quad\quad\quad+ ik\inner{G_e^{(1)}(\un\times\mathbf{u})}{\mathbf{u}_T'}_{\partial B_R},
\end{align*}
for all $(\mathbf{u},\mathbf{v},p),(\mathbf{u}',\mathbf{v}',p')\in \boldsymbol{\mathcal{X}}_0$. We see that
\begin{equation*}
((\mathbb{A}+\mathbb{B})(\mathbf{u},\mathbf{v},p),(\mathbf{u}',\mathbf{v}',p'))_{\boldsymbol{\mathcal{X}}_0} = a((\mathbf{u},\mathbf{v},p),(\mathbf{u}',\mathbf{v}',p')) \quad\quad\quad\quad\forall(\mathbf{u},\mathbf{v},p),(\mathbf{u}',\mathbf{v}',p')\in \boldsymbol{\mathcal{X}}_0,
\end{equation*}
and as a result we need only study the operators $\mathbb{A}$ and $\mathbb{B}$. It is clear from the definition of $\mathbb{A}$ and Remark \ref{remark_Ge} that
\begin{equation*}
\abs{(\mathbb{A}(\mathbf{u},\mathbf{v},p),(\mathbf{u}',\mathbf{v}',p'))_{\boldsymbol{\mathcal{X}}_0}} \ge C\left( \norm{\mathbf{u}}_{\Hcurl{B_R\setminus\overline{B}}}^2 + \norm{\mathbf{v}}_{\Hcurl{B}}^2 + \norm{\nabla p}_B^2 \right),
\end{equation*}
and it follows from the Lax-Milgram lemma that $\mathbb{A}:\boldsymbol{\mathcal{X}}_0\to \boldsymbol{\mathcal{X}}_0$ is invertible with bounded inverse. The compactness of $\mathbb{B}:\boldsymbol{\mathcal{X}}_0\to \boldsymbol{\mathcal{X}}_0$ follows easily from Theorem \ref{theorem_X0} and Remark \ref{remark_Ge}, and consequently we see that the operator $\mathbb{A} + \mathbb{B}$ is a Fredholm operator of index zero. Therefore, we conclude that the auxiliary problem \eqref{aux} is of Fredholm type, and since we already showed uniqueness of solutions in Theorem \ref{theorem_uniqueness} we obtain the following result. We remark that the proof of well-posedness may also be approached using the limiting absorption principle as in \cite{nguyen2018}.

\begin{theorem} \label{theorem_aux}

If $\Im(\eta)\ge0$, then there exists a unique solution of \eqref{aux} with the estimate
\begin{equation*}
\norm{\mathbf{E}_0^s}_{\Hcurl{B_R\setminus\overline{B}}} + \norm{\mathbf{E}_0}_{\Hcurl{B}} + \norm{\nabla P}_B \le C_R\left( \norm{\un\times\mathbf{E}^i}_{\mathbf{H}^{-1/2}(\Div,\partial B)} + \norm{\un\times\curl\mathbf{E}^i}_{\mathbf{H}^{-1/2}(\Div,\partial B)} \right),
\end{equation*}
where the constant $C_R$ is independent of $\mathbf{E}^i$ but depends on the domain $B_R$.

\end{theorem}

Now that we have established that the auxiliary problem is well-posed, we turn our attention to the physical and auxiliary data that will be collected in order to determine the eigenvalues corresponding to a medium. If we choose a plane wave incident field with direction $\mathbf{d}\in\mathbb{S}^2$ and polarization $\mathbf{p}\in\mathbb{R}^3\setminus\{\mathbf{0}\}$ given by
\begin{equation*}
\mathbf{E}^i(\mathbf{x},\mathbf{d};\mathbf{p}) := \frac{i}{k}\curl\curl\mathbf{p} e^{ik\mathbf{x}\cdot\mathbf{d}},
\end{equation*}
then the scattered field for both the physical and auxiliary scattering problems (\eqref{sc} and \eqref{aux}, respectively) has the asymptotic behavior
\begin{align*}
\mathbf{E}^s(\mathbf{x}) &= \frac{e^{ik\abs{\mathbf{x}}}}{\abs{\mathbf{x}}}\left[ \mathbf{E}_\infty(\hat{\mathbf{x}},\mathbf{d};\mathbf{p}) + \mathcal{O}\left(\frac{1}{\abs{\mathbf{x}}}\right) \right], \quad\abs{\mathbf{x}}\to\infty, \\
\mathbf{E}_0^s(\mathbf{x}) &= \frac{e^{ik\abs{\mathbf{x}}}}{\abs{\mathbf{x}}}\left[ \mathbf{E}_{0,\infty}(\hat{\mathbf{x}},\mathbf{d};\mathbf{p}) + \mathcal{O}\left(\frac{1}{\abs{\mathbf{x}}}\right) \right], \quad\abs{\mathbf{x}}\to\infty,
\end{align*}
where for a nonzero $\mathbf{x}\in\mathbb{R}^3$ we define $\hat{\mathbf{x}} := \frac{\mathbf{x}}{\abs{\mathbf{x}}}$. The amplitudes $\mathbf{E}_\infty(\hat{\mathbf{x}},\mathbf{d};\mathbf{p})$ and $\mathbf{E}_{0,\infty}(\hat{\mathbf{x}},\mathbf{d};\mathbf{p})$ are the electric far field patterns of the physical and auxiliary problems, respectively, and they serve as the data for our problem; the physical far field pattern is collected from the system under investigation, and the auxiliary far field pattern is computed for various values of the parameter $\eta$. We remark that from standard arguments (cf. \cite{monk}) it follows that these electric far field patterns satisfy the reciprocity relations
\begin{align*}
\mathbf{q}\cdot\mathbf{E}_\infty(\hat{\mathbf{x}},\mathbf{d};\mathbf{p}) &= \mathbf{p}\cdot\mathbf{E}_\infty(-\mathbf{d},-\hat{\mathbf{x}};\mathbf{q}), \\
\mathbf{q}\cdot\mathbf{E}_{0,\infty}(\hat{\mathbf{x}},\mathbf{d};\mathbf{p}) &= \mathbf{p}\cdot\mathbf{E}_{0,\infty}(-\mathbf{d},-\hat{\mathbf{x}};\mathbf{q}),
\end{align*}
for all $\hat{\mathbf{x}},\mathbf{d}\in\mathbb{S}^2$ and $\mathbf{p},\mathbf{q}\in\mathbb{R}^3$.

We now introduce the electric far field operator $\mathbf{F}:\mathbf{L}_t^2(\mathbb{S}^2)\to \mathbf{L}_t^2(\mathbb{S}^2)$ defined by
\begin{equation*}
(\mathbf{F}\mathbf{g})(\hat{\mathbf{x}}) := \int_{\mathbb{S}^2} \mathbf{E}_\infty(\hat{\mathbf{x}},\mathbf{d};\mathbf{g}(\mathbf{d})) \, ds(\mathbf{d}), \; \hat{\mathbf{x}}\in\mathbb{S}^2,
\end{equation*}
and we introduce the auxiliary far field operator $\mathbf{F}_0:\mathbf{L}_t^2(\mathbb{S}^2)\to \mathbf{L}_t^2(\mathbb{S}^2)$ defined by
\begin{equation*}
(\mathbf{F}_0\mathbf{g})(\hat{\mathbf{x}}) := \int_{\mathbb{S}^2} \mathbf{E}_{0,\infty}(\hat{\mathbf{x}},\mathbf{d};\mathbf{g}(\mathbf{d})) \, ds(\mathbf{d}), \; \hat{\mathbf{x}}\in\mathbb{S}^2.
\end{equation*}
With these definitions in hand, we further define the modified far field operator $\boldsymbol{\mathcal{F}}:\mathbf{L}_t^2(\mathbb{S}^2)\to \mathbf{L}_t^2(\mathbb{S}^2)$ by $\boldsymbol{\mathcal{F}} := \mathbf{F} - \mathbf{F}_0$, which may be written explicitly as
\begin{equation*}
(\boldsymbol{\mathcal{F}}\mathbf{g})(\hat{\mathbf{x}}) := \int_{\mathbb{S}^2} \Bigr[ \mathbf{E}_\infty(\hat{\mathbf{x}},\mathbf{d};\mathbf{g}(\mathbf{d})) - \mathbf{E}_{0,\infty}(\hat{\mathbf{x}},\mathbf{d};\mathbf{g}(\mathbf{d})) \Bigr] \, ds(\mathbf{d}), \; \hat{\mathbf{x}}\in\mathbb{S}^2.
\end{equation*}
The modified far field operator serves to compare the electric far field patterns of the physical and auxiliary problems, and in order to generate an eigenvalue problem we characterize when this operator is injective. We state the following theorem, which follows in a similar manner to \cite[Theorem 4.14]{cakoni_colton_monk} and \cite[Section 4]{camano_lackner_monk}.

\begin{theorem} \label{mffo}

The modified far field operator $\boldsymbol{\mathcal{F}}$ is injective with dense range if and only if there does not exist a nontrivial solution $(\mathbf{w},\mathbf{v},p)$ of the modified interior transmission problem
\begin{subequations} \label{mit}
\begin{align}
\curl\curl\mathbf{w} - k^2\epsilon\mathbf{w} &= \mathbf{0} \text{ in } B, \label{mit1} \\
\curl\gamma^{-1}\curl\mathbf{v} - k^2\eta\mathbf{v} + k^2\nabla p &= \mathbf{0} \text{ in } B, \label{mit2} \\
\div\mathbf{v} &= 0 \text{ in } B, \label{mit3} \\
\un\cdot\mathbf{v} &= 0 \text{ on } \partial B, \label{mit4} \\
\un\times(\mathbf{w}-\mathbf{v}) &= \mathbf{0} \text{ on } \partial B, \label{mit5} \\
\un\times\left(\curl\mathbf{w} - \gamma^{-1}\curl\mathbf{v}\right) &= \mathbf{0} \text{ on } \partial B, \label{mit6}
\end{align}
\end{subequations}
for which $\mathbf{v}$ and $p$ are of the form
\begin{equation*} \mathbf{v}(\mathbf{x}) = \int_{\mathbb{S}^2} \mathbf{E}_0(\mathbf{x},\mathbf{d};\mathbf{g}(\mathbf{d})) \, ds(\mathbf{d}), \quad p(\mathbf{x}) = \int_{\mathbb{S}^2} P(\mathbf{x},\mathbf{d};\mathbf{g}(\mathbf{d})) \, ds(\mathbf{d}), \;\mathbf{x}\in B, \end{equation*}
where $\mathbf{E}_0(\cdot,\mathbf{d};\mathbf{p})$ and $P(\cdot,\mathbf{d};\mathbf{p})$ satisfy \eqref{aux} with a plane wave incident field $\mathbf{E}^i(\cdot,\mathbf{d};\mathbf{p})$.

\end{theorem}

For a fixed $\gamma$, we call a value of $\eta$ for which there exists a nontrivial solution $(\mathbf{w},\mathbf{v},p)$ of \eqref{mit} a \emph{modified electromagnetic transmission eigenvalue}, and in the following sections we investigate this problem and its eigenvalues in greater detail. For future reference we define an \emph{electromagnetic Herglotz wave function} $\mathbf{v}_\mathbf{g}^i$ by
\begin{equation} \label{herglotz}
\mathbf{v}_\mathbf{g}^i(\mathbf{x}) := ik\int_{\mathbb{S}^2} e^{-ik\mathbf{x}\cdot\mathbf{d}} \mathbf{g}(\mathbf{d}) ds(\mathbf{d}), \quad \mathbf{x}\in\mathbb{R}^3.
\end{equation}
We remark that by linearity the electric far field pattern of the physical problem \eqref{sc} for $\mathbf{u}^i = \mathbf{v}_\mathbf{g}^i$ is given by $\mathbf{F}\mathbf{g}$, and the same relationship holds between the far field pattern of the auxiliary problem \eqref{aux} and the auxiliary far field operator $\mathbf{F}_0$.

\section{The modified interior transmission problem} \label{sec_mitp}

We now study a nonhomogeneous version of the modified interior transmission problem \eqref{mit}, which is to find 
$(\mathbf{w},\mathbf{v},p)\in \Hcurl{B}\times\Hcurl{B}\times H_*^1(B)$ satisfying
\begin{subequations} \label{nmit}
\begin{align}
\curl\curl\mathbf{w} - k^2\epsilon\mathbf{w} &= \mathbf{f} \text{ in } B, \label{nmit1} \\
\curl\gamma^{-1}\curl\mathbf{v} - k^2\eta\mathbf{v} + k^2\nabla p &= \mathbf{g} \text{ in } B, \label{nmit2} \\
\div\mathbf{v} &= 0 \text{ in } B, \label{nmit3} \\
\un\cdot\mathbf{v} &= 0 \text{ on } \partial B, \label{nmit4} \\
\un\times(\mathbf{w}-\mathbf{v}) &= \boldsymbol{\xi} \text{ on } \partial B, \label{nmit5} \\
\un\times\left(\curl\mathbf{w} - \gamma^{-1}\curl\mathbf{v}\right) &= \mathbf{h} \text{ on } \partial B, \label{nmit6} 
\end{align}
\end{subequations}
where $\mathbf{f}\in\mathbf{L}^2(B)$, $\mathbf{g}\in\Hdivzero{B}$, and $\boldsymbol{\xi},\mathbf{h}\in\mathbf{H}^{-1/2}(\textnormal{Div},B)$ are given. Here we have used $\Hdivzero{B}$ to denote the subspace of $\mathbf{L}^2(B)$ consisting of vector fields with vanishing divergence in $B$. Our approach will be similar to our analysis of the auxiliary problem in Section \ref{sec_aux}, and in the remark following the next assumption we make two modifications analogous to those of Remark \ref{remark_mod1}.

\begin{assumption} \label{assumption_dirichlet}
\textnormal{
We assume that $k$ is chosen such that there exists a unique
$\boldsymbol{\varphi}\in\Hcurl{B}$ satisfying
\begin{subequations} \label{dir}
\begin{align}
\curl\curl\boldsymbol{\varphi} - k^2\epsilon\boldsymbol{\varphi} &= \mathbf{f} \text{ in } B, \label{dir1} \\
\un\times\boldsymbol{\varphi} &= \boldsymbol{\xi} \text{ on } \partial B, \label{dir2}
\end{align}
\end{subequations}
with the estimate
\begin{equation} \label{phi_est}
\norm{\boldsymbol{\varphi}}_{\Hcurl{B}} \le C\left(\norm{\mathbf{f}}_{\mathbf{L}^2(B)} + \norm{\boldsymbol{\xi}}_{\mathbf{H}^{-1/2}(\textnormal{Div},\partial B)}\right)
\end{equation}
for a constant $C>0$ independent of $\mathbf{f}$ and $\boldsymbol{\xi}$.
}
\end{assumption}

\begin{remark} \label{remark_mod2}
\textnormal{
Assumption \ref{assumption_dirichlet} holds provided that $k^2$ is not an interior Maxwell eigenvalue (cf. \cite[Chapter 4]{monk}. Under this assumption we choose a lifting function $\boldsymbol{\varphi}\in\Hcurl{B}$ to be the unique solution of \eqref{dir} with the estimate \eqref{phi_est}. We may now replace $\mathbf{w}$ with $\mathbf{w} - \boldsymbol{\varphi}$ and modify $\mathbf{h}$ accordingly to obtain \eqref{nmit} with $\mathbf{f} = \mathbf{0}$ and $\boldsymbol{\xi} = \mathbf{0}$. Rather than make this modification explicit, we assume without loss of generality that $\mathbf{f} = \mathbf{0}$ and $\boldsymbol{\xi} = \mathbf{0}$. For the second modification we define $\zeta\in H_*^1(B)$ as the unique solution of
\begin{align*}
\Delta\zeta &= 0 \text{ in } B, \\
\normal{\zeta} &= k^{-2}\left(  -\un\cdot\mathbf{g} + \nabla_{\partial B}\cdot\mathbf{h} \right) \text{ on } \partial B.
\end{align*}
We replace $p$ with $p+\zeta$ in \eqref{nmit2}, which (along with our assumption that $\mathbf{f} = \mathbf{0}$ and $\boldsymbol{\xi} = \mathbf{0}$ from Remark \ref{remark_mod2}) results in the equivalent problem of finding
$(\mathbf{w},\mathbf{v},p)\in \Hcurl{B}\times\Hcurl{B}\times H_*^1(B)$ satisfying
\begin{subequations} \label{eq_nmit}
\begin{align}
\curl\curl\mathbf{w} - k^2\epsilon\mathbf{w} &= \mathbf{0} \text{ in } B, \label{eq_nmit1} \\
\curl\gamma^{-1}\curl\mathbf{v} - k^2\eta\mathbf{v} + k^2\nabla p &= \mathbf{g} + k^2\nabla\zeta \text{ in } B, \label{eq_nmit2} \\
\div\mathbf{v} &= 0 \text{ in } B, \label{eq_nmit3} \\
\un\cdot\mathbf{v} &= 0 \text{ on } \partial B. \label{eq_nmit4} \\
\un\times(\mathbf{w}-\mathbf{v}) &= \mathbf{0} \text{ on } \partial B, \label{eq_nmit5} \\
\un\times\left(\curl\mathbf{w} - \gamma^{-1}\curl\mathbf{v}\right) &= \mathbf{h} \text{ on } \partial B, \label{eq_nmit6}
\end{align}
\end{subequations}
}
\end{remark}

We are now in a position to study the nonhomogeneous problem \eqref{nmit} through the equivalent problem \eqref{eq_nmit}. We first define the space
\begin{equation*} \boldsymbol{\mathcal{H}} := \{(\mathbf{w},\mathbf{v},p)\in \Hcurl{B}\times\Hcurl{B}\times H_*^1(B) \mid \mathbf{w}-\mathbf{v}\in\Hcurltr{B}\}, \end{equation*}
equipped with the standard inner product on $\Hcurl{B}\times\Hcurl{B}\times H_*^1(B)$, and we see that an equivalent variational formulation of \eqref{eq_nmit} is to find $(\mathbf{w},\mathbf{v},p)\in\boldsymbol{\mathcal{H}}$ such that
\begin{align}
\begin{split} \label{varprob}
(\curl\mathbf{w},\curl\mathbf{w}')_B &- \gamma^{-1}(\curl\mathbf{v},\curl\mathbf{v}')_B - k^2(\epsilon\mathbf{w},\mathbf{w}')_B + k^2\eta(\mathbf{v},\mathbf{v}')_B \\
&- k^2(\nabla p,\mathbf{v}')_B - k^2(\mathbf{v},\nabla p')_B = -(\mathbf{g},\mathbf{v}')_B - \inner{\mathbf{h}}{\mathbf{w}_T'}_{\partial B} - k^2(\nabla\zeta,\mathbf{v}')_B
\end{split}
\end{align}
for all $(\mathbf{w}',\mathbf{v}',p')\in\boldsymbol{\mathcal{H}}$. If we choose $(\mathbf{w}',\mathbf{v}',p') = (\mathbf{0},\mathbf{0},q)$ in \eqref{varprob} for any $q\in H_*^1(B)$, then we see that
\begin{equation}
(\mathbf{v},\nabla q)_B = 0 \quad\forall q\in H_*^1(B), \label{var_prop1}
\end{equation}
which implies that $\div\mathbf{v} = 0$ in $B$ and $\un\cdot\mathbf{v} = 0$ on $\partial B$. By choosing $(\mathbf{w}',\mathbf{v}',p') = (\varphi,\psi,q)$ for $\varphi,\psi,q\in H_*^1(B)$ such that $\varphi-\psi\in H_0^1(B)$ (which implies that $\un\times(\nabla\varphi-\nabla\psi) = \mathbf{0}$ on $\partial B$) in \eqref{varprob} and applying the previous result we see that
\begin{equation}
k^2(\epsilon\mathbf{w},\nabla\varphi)_B + k^2(\nabla p,\nabla \psi)_B = -(\mathbf{g},\nabla\psi)_B - \inner{\mathbf{h}}{\nabla_{\partial B}\varphi}_{\partial B} - k^2(\nabla\zeta,\nabla\psi)_B. \label{var_prop2a}
\end{equation}
However, by the divergence theorem and the definition of the surface divergence operator the right-hand side of \eqref{var_prop2a} becomes
\begin{align*}
- (\mathbf{g},\nabla\psi)_B - \inner{\mathbf{h}}{\nabla_{\partial B}\varphi}_{\partial B} - k^2(\nabla\zeta,\nabla\psi)_B &= -\inner{\un\cdot\mathbf{g}}{\psi}_{\partial B} + \inner{\nabla_{\partial B}\cdot\mathbf{h}}{\varphi}_{\partial B} - k^2\inner{\normal{\zeta}}{\psi}_{\partial B} \\
&= \inner{\left(-\un\cdot\mathbf{g} + \nabla_{\partial B}\cdot\mathbf{h}\right) - k^2\normal{\zeta}}{\varphi}_{\partial B}
\\&= 0,
\end{align*}
where the final equality follows from the definition of $\zeta$ in Remark \ref{remark_mod2}. We conclude that
\begin{equation}
(\epsilon\mathbf{w},\nabla\varphi)_B + (\nabla p,\nabla \psi)_B = 0. \label{var_prop2b}
\end{equation}
This result motivates us to define the spaces
\begin{gather*}
S := \{(\varphi,\psi,q)\in(H_*^1(B))^3 \mid \varphi-\psi\in H_0^1(B)\}, \\
\boldsymbol{\mathcal{H}}_0 := \{(\mathbf{w},\mathbf{v},p)\in\boldsymbol{\mathcal{H}} \mid (\epsilon\mathbf{w},\nabla\varphi)_B + (\nabla p,\nabla \psi)_B + (\mathbf{v},\nabla q)_B = 0 \quad\forall(\varphi,\psi,q)\in S\},
\end{gather*}
where $\boldsymbol{\mathcal{H}}_0$ is equipped with the inner product on $\boldsymbol{\mathcal{H}}$. We observe that the space $\boldsymbol{\mathcal{H}}_0$ includes both conditions \eqref{var_prop1} and \eqref{var_prop2b} that must be satisfied by solutions of the modified interior transmission problem \eqref{mit}, and we will use this fact in order to establish results on the solvability of this problem. In the following lemma we clarify the condition built into the definition of the space $\boldsymbol{\mathcal{H}}_0$.

\begin{lemma} \label{lemma_H0}

A given $(\mathbf{w},\mathbf{v},p)\in\boldsymbol{\mathcal{H}}$ lies in the space $\boldsymbol{\mathcal{H}}_0$ if and only if 
\begin{subequations} \label{H0}
\begin{gather}
\div(\epsilon\mathbf{w}) = 0, \quad \Delta p = 0, \quad \div\mathbf{v} = 0 \text{ in } B, \label{H01} \\
\normal{p} + \un\cdot\epsilon\mathbf{w} = 0, \quad \un\cdot\mathbf{v} = 0 \text{ on } \partial B. \label{H02}
\end{gather}
\end{subequations}

\end{lemma}

\begin{proof}

If $(\mathbf{w},\mathbf{v},p)\in\boldsymbol{\mathcal{H}}_0$, then we already established the conditions on $\mathbf{v}$ in \eqref{H0}. If we choose $(\varphi,\psi,\xi) = (\phi,0,0)$ for some $\phi\in H_0^1(B)$, then it follows that $\div(\epsilon\mathbf{w}) = 0$ in $B$, and similar reasoning implies that $\Delta p = 0$ in $B$. For all $(\varphi,\psi,0)\in S$ we see from the divergence theorem and the preceding results that
\begin{equation*}
\inner{\normal{p} + \un\cdot\epsilon\mathbf{w}}{\varphi}_{\partial B} = \inner{\un\cdot\epsilon\mathbf{w}}{\varphi}_{\partial B} + \inner{\normal{p}}{\psi}_{\partial B} = (\epsilon\mathbf{w},\nabla\varphi)_B + (\nabla p,\nabla\psi)_B = 0,
\end{equation*}
which provides the remaining condition in \eqref{H02}. \par

Conversely, if $(\mathbf{w},\mathbf{v},p)$ satisfies \eqref{H0}, then for all $(\varphi,\psi,\xi)\in S$ it follows from the divergence theorem that
\begin{equation*}
(\epsilon\mathbf{w},\nabla\varphi)_B + (\nabla p,\nabla \psi)_B + (\nabla\xi,\mathbf{v})_B = \inner{\un\cdot\epsilon\mathbf{w}}{\varphi}_{\partial B} + \inner{\normal{p}}{\psi}_{\partial B} = \inner{\un\cdot\epsilon\mathbf{w} + \normal{p}}{\varphi}_{\partial B} = 0,
\end{equation*}
and we conclude that $(\mathbf{w},\mathbf{v},p)\in\boldsymbol{\mathcal{H}}_0$. \proofend

\end{proof}

We will return to the space $\boldsymbol{\mathcal{H}}_0$, but we must first introduce an important method for establishing solvability results for \eqref{nmit} when $\gamma\neq1$. In the variational formulation \eqref{varprob}, we see that the principal part of the associated operator is sign-indefinite, and as a result we appeal to $\mathcal{T}$-coercivity in order to restore positivity (cf. \cite{cakoni_colton_haddar,chesnel}). In particular, in the case $0<\gamma<1$ we define the operator $\mathcal{T}:\boldsymbol{\mathcal{H}}\to\boldsymbol{\mathcal{H}}$ by $\mathcal{T}(\mathbf{w},\mathbf{v},p) := (\mathbf{w}-2\mathbf{v},-\mathbf{v},p)$, and we see that $\mathcal{T}^2 = I$ and consequently that $\mathcal{T}$ is an isomorphism. We will provide proofs for this case and then simply state the appropriate choice of $\mathcal{T}$ for the case $\gamma>1$. If the sesquilinear form $a_\eta(\cdot,\cdot)$ is defined on $\boldsymbol{\mathcal{H}}\times\boldsymbol{\mathcal{H}}$ by the left-hand side of \eqref{varprob}, then we define the sesquilinear form $a_\eta^\mathcal{T}(\cdot,\cdot)$ by
\begin{equation*}
a_\eta^\mathcal{T}((\mathbf{w},\mathbf{v},p),(\mathbf{w}',\mathbf{v}',p')) := a_\eta((\mathbf{w},\mathbf{v},p),\mathcal{T}(\mathbf{w}',\mathbf{v}',p')) \quad\forall(\mathbf{w},\mathbf{v},p),(\mathbf{w}',\mathbf{v}',p')\in\boldsymbol{\mathcal{H}}.
\end{equation*}
Although we will see that we have restored positivity of the principal part of the operator with the introduction of $\mathcal{T}$, the problem still remains that the space $\boldsymbol{\mathcal{H}}$ is not compactly embedded into $\mathbf{L}^2(B)\times\mathbf{L}^2(B)\times H_*^1(B)$. However, we may obtain compactness by working in the space $\boldsymbol{\mathcal{H}}_0$, which is still a valid space in which to seek the solution since $\mathcal{T}$ is an isomorphism on $\boldsymbol{\mathcal{H}}$. Thus, we introduce the problem of finding $(\mathbf{w},\mathbf{v},p)\in\boldsymbol{\mathcal{H}}_0$ which satisfies
\begin{equation}
a_\eta^\mathcal{T}((\mathbf{w},\mathbf{v},p),(\mathbf{w}',\mathbf{v}',p')) = \ell(\mathcal{T}(\mathbf{w}',\mathbf{v}',p')) \quad\forall(\mathbf{w}',\mathbf{v}',p')\in\boldsymbol{\mathcal{H}}_0, \label{varprob_T}
\end{equation}
where $\ell$ is the antilinear functional on $\boldsymbol{\mathcal{H}}_0$ representing the right-hand sides in \eqref{eq_nmit} with the isomorphism $\mathcal{T}$ applied to the test functions. By means of the Riesz representation theorem we define the operator $A_\eta^\mathcal{T}:\boldsymbol{\mathcal{H}}_0\to\boldsymbol{\mathcal{H}}_0$ such that
\begin{equation*}
(A_\eta^\mathcal{T}(\mathbf{w},\mathbf{v},p),(\mathbf{w}',\mathbf{v}',p'))_{\boldsymbol{\mathcal{H}}} = a_\eta^\mathcal{T}((\mathbf{w},\mathbf{v},p),(\mathbf{w}',\mathbf{v}',p')) \quad\forall(\mathbf{w},\mathbf{v},p),(\mathbf{w}',\mathbf{v}',p')\in\boldsymbol{\mathcal{H}}_0.
\end{equation*}
We observe that if $(\mathbf{w},\mathbf{v},p)$ satisfies \eqref{eq_nmit} for some $\eta\in\mathbb{C}$, then $A_\eta^\mathcal{T}(\mathbf{w},\mathbf{v},p) = \mathbf{L}$, where $\mathbf{L}\in\boldsymbol{\mathcal{H}}_0$ is the Riesz representer for $\ell\circ\mathcal{T}$ in $\boldsymbol{\mathcal{H}}_0$, and as a result we study the operator $A_\eta^\mathcal{T}$. We must first prove a compactness result for the space $\boldsymbol{\mathcal{H}}_0$, and for that we will again appeal to Theorem \ref{theorem_costabel} along with another compactness result from \cite{monk}.

%
%

\begin{theorem} \label{theorem_compact}

The space $\boldsymbol{\mathcal{H}}_0$ is compactly embedded into $\mathbf{L}^2(B)\times\mathbf{L}^2(B)\times H_*^1(B)$.

\end{theorem}

\begin{proof}

We let $\{(\mathbf{w}_m,\mathbf{v}_m,p_m)\}$ be a bounded sequence in $\boldsymbol{\mathcal{H}}_0$, and we observe that $\{\mathbf{v}_m\}$ is a bounded sequence in $\Hcurl{B}$. Since $\div\mathbf{v}_m = 0 \in L^2(B)$ and $\un\cdot\mathbf{v}_m = \mathbf{0} \in \mathbf{L}^2(\partial B)$ it follows from Theorem \ref{theorem_costabel} that $\{\un\times\mathbf{v}_m\}$ is a bounded sequence in $\mathbf{L}_t^2(\partial B)$ and $\{\mathbf{v}_m\}$ is a bounded sequence in $\mathbf{H}^{1/2}(B)$. In particular, the compact embedding of $\mathbf{H}^{1/2}(B)$ into $\mathbf{L}^2(B)$ implies the existence of a subsequence of $\{\mathbf{v}_m\}$ that converges in the latter space. By passing to the corresponding subsequence of $\{\mathbf{w}_m\}$ without changing notation, we also see that $\{\mathbf{w}_m\}$ is a bounded sequence in the space
\begin{equation*}
\mathbf{X}_0 := \left\{\boldsymbol{\psi}\in\Hcurl{B} \mid \div(\epsilon\boldsymbol{\psi}) = 0 \text{ in } B, \,\un\times\boldsymbol{\psi}\in\mathbf{L}_t^2(\partial B)\right\},
\end{equation*}
equipped with the norm defined by
\begin{equation*}
\norm{\boldsymbol{\psi}}_{\mathbf{X}_0}^2 := \norm{\boldsymbol{\psi}}_{\Hcurl{B}}^2 + \norm{\un\times\boldsymbol{\psi}}_{\mathbf{L}_t^2(\partial B)}^2,
\end{equation*}
which follows from the fact that $\un\times\mathbf{w}_m = \un\times\mathbf{v}_m$ on $\partial B$. The compact embedding of the space $\mathbf{X}_0$ into $\mathbf{L}^2(B)$ (cf. \cite[Theorem 4.7]{monk}) allows us to conclude that there exists a subsequence $\{\mathbf{w}_m\}$ that converges in $\mathbf{L}^2(B)$. Finally, we see that each $p_m\in H_*^1(B)$ satisfies the well-posed Neumann problem
\begin{align*}
\Delta p_m &= 0 \text{ in } B, \\
\normal{p_m} &= -\un\cdot(\epsilon\mathbf{w}_m) \text{ on } \partial B,
\end{align*}
which implies the existence of a constant $C$ independent of $m$ such that
\begin{equation}
\norm{\nabla p_m}_{\mathbf{L}^2(B)} \le C\norm{\un\cdot(\epsilon\mathbf{w}_m)}_{H^{-1/2}(\partial B)}. \label{neumann_estimate}
\end{equation}
Since the sequence $\{\epsilon\mathbf{w}_m\}$ lies in the space $\Hdivzero{B}$ and contains a convergent subsequence in this space, the normal trace theorem for the space $\Hdiv{B}$ and the estimate \eqref{neumann_estimate} imply that the corresponding subsequence of $\{p_m\}$ converges in $H_*^1(B)$. Therefore, we conclude that a subsequence of $\{(\mathbf{w}_m,\mathbf{v}_m,p_m)\}$ converges in the space $\mathbf{L}^2(B)\times\mathbf{L}^2(B)\times H_*^1(B)$. \proofend

\end{proof}

With this compactness result, we may now split the operator $A_\eta^\mathcal{T}$ into an invertible and compact part, and to this end we define the operators $\hat{A}^\mathcal{T}, B_\eta^\mathcal{T}:\boldsymbol{\mathcal{H}}_0\to\boldsymbol{\mathcal{H}}_0$ such that
\begin{align*}
(\hat{A}^\mathcal{T}(\mathbf{w},\mathbf{v},p),(\mathbf{w}',\mathbf{v}',p'))_{\boldsymbol{\mathcal{H}}_0} &= (\curl\mathbf{w},\curl\mathbf{w}')_B + \gamma^{-1}(\curl\mathbf{v},\curl\mathbf{v}')_B + k^2(\mathbf{w},\mathbf{w}')_B \\
&\quad\quad+ k^2(\mathbf{v},\mathbf{v}')_B - 2(\curl\mathbf{w},\curl\mathbf{v}')_B + (\nabla p,\nabla p')_B, \\
(B_\eta^\mathcal{T}(\mathbf{w},\mathbf{v},p),(\mathbf{w}',\mathbf{v}',p'))_{\boldsymbol{\mathcal{H}}_0} &= -k^2((\epsilon+1)\mathbf{w},\mathbf{w}')_B - k^2(\eta+1)(\mathbf{v},\mathbf{v}')_B + 2k^2(\epsilon\mathbf{w},\mathbf{v}')_B
\end{align*}
for all $(\mathbf{w},\mathbf{v},p)\in\boldsymbol{\mathcal{H}}_0$. From the definition of $A_\eta^\mathcal{T}$ we see that $A_\eta^\mathcal{T} = \hat{A}^\mathcal{T} + B_\eta^\mathcal{T}$. An application of Young's inequality implies that
\begin{equation*}
2\abs{(\curl\mathbf{w},\curl\mathbf{v})_B} \le \delta(\curl\mathbf{w},\curl\mathbf{w})_B + \delta^{-1}(\curl\mathbf{v},\curl\mathbf{v})_B
\end{equation*}
for any $\delta>0$, and it follows that
\begin{align*}
\abs{(\hat{A}^\mathcal{T}(\mathbf{w},\mathbf{v},p),(\mathbf{w},\mathbf{v},p))_{\boldsymbol{\mathcal{H}}_0}} &\ge (1-\delta)(\curl\mathbf{w},\curl\mathbf{w})_B + \left(\gamma^{-1} - \delta^{-1}\right)(\curl\mathbf{v},\curl\mathbf{v})_B \\
&\quad\quad+ k^2(\mathbf{w},\mathbf{w})_B+ k^2(\mathbf{v},\mathbf{v})_B + (\nabla p,\nabla p)_B.
\end{align*}
If $0<\gamma<1$ and we choose $\delta\in(\gamma,1)$, then we conclude from the Lax-Milgram Lemma that $\hat{A}^\mathcal{T}$ is invertible with bounded inverse. In the case $\gamma>1$, we may use the isomorphism defined by $\mathcal{T}(\mathbf{w},\mathbf{v},p) := (\mathbf{w},-\mathbf{v}+2\mathbf{w},p)$ in the same manner to conclude invertibility of $\hat{A}^\mathcal{T}$. In either case, the operator $B_\eta^\mathcal{T}$ is compact, as can be easily seen from the compact embedding of $\boldsymbol{\mathcal{H}}_0$ into $\mathbf{L}^2(B)\times\mathbf{L}^2(B)\times H_*^1(B)$ that we established in Theorem \ref{theorem_compact}. Therefore, we have shown that the operator $A_\eta^\mathcal{T} = \hat{A}^\mathcal{T} + B_\eta^\mathcal{T}$ is Fredholm of index zero provided that $\gamma\neq1$, and since $\mathcal{T}$ is an isomorphism we obtain the following result. We remark that, although we chose to use $\mathcal{T}$-coercivity in the present discussion, the proof of the subsequent theorem may alternatively be approached using techniques from \cite{cakoni_nguyen,nguyen_sil}.

\begin{theorem} \label{theorem_Fredholm}

If $\gamma\neq1$, then the modified interior transmission problem \eqref{nmit} is of Fredholm type. In particular, if $\eta$ is not a modified transmission eigenvalue, then there exists a unique solution $(\mathbf{w},\mathbf{v},p)\in\Hcurl{B}\times\Hcurl{B}\times H_*^1(B)$ of \eqref{nmit} satisfying the estimate
\begin{equation}
\norm{\mathbf{w}}_{\mathbf{H}(\curl,B)} + \norm{\mathbf{v}}_{\mathbf{H}(\curl,B)} + \norm{\nabla p}_B \le C\Bigr( \norm{\mathbf{f}}_B + \norm{\mathbf{g}}_B + \norm{\boldsymbol{\xi}}_{\mathbf{H}^{-1/2}(\textnormal{Div},\partial B)} + \norm{\mathbf{h}}_{\mathbf{H}^{-1/2}(\textnormal{Div},\partial B)} \Bigr). \label{fredholm_est}
\end{equation}

\end{theorem}

We remark that the immediate result holds only for the equivalent problem \eqref{eq_nmit}, and Theorem \ref{theorem_Fredholm} follows from standard arguments on lifting functions since both $\boldsymbol{\varphi}$ and $\zeta$ defined in Remark \ref{remark_mod2} are controlled by the appropriate norms of the right-hand sides of \eqref{nmit}. The solvability result we established in Theorem \ref{theorem_Fredholm} and its preceding arguments will allow us to study the class of modified electromagnetic transmission eigenvalues in the next section.

\section{Properties of modified electromagnetic transmission eigenvalues}

\label{sec_props}

We now investigate the properties of modified transmission eigenvalues, and we begin with an application of the analytic Fredholm theorem (cf. \cite[Theorem 8.26]{colton_kress}). Since the mapping $\eta\mapsto B_\eta^\mathcal{T}$ is clearly analytic, this theorem asserts that either i) the operator $A_\eta^\mathcal{T} = \hat{A}^\mathcal{T} + B_\eta^\mathcal{T}$ is invertible for no values of $\eta$, or ii) the operator $A_\eta^\mathcal{T} = \hat{A}^\mathcal{T} + B_\eta^\mathcal{T}$ is invertible for all $\eta$ except possibly in a discrete subset of the complex plane. Our aim is to show that the second statement holds, which will follow once we establish the existence of at least one value of $\eta$ for which $A_\eta^\mathcal{T}$ is injective, as this property implies invertibility by Theorem \ref{theorem_Fredholm}. \par

We suppose that $(\mathbf{w},\mathbf{v},p)$ satisfies \eqref{mit} for some $\eta\in\mathbb{C}$, which may be written variationally as
\begin{align*}
(\curl\mathbf{w},\curl\mathbf{w}')_B - \gamma^{-1}(\curl\mathbf{v},&\curl\mathbf{v}')_B - k^2(\epsilon\mathbf{w},\mathbf{w}')_B + k^2\eta(\mathbf{v},\mathbf{v}')_B  \nonumber \\
&- k^2(\nabla p,\mathbf{v}')_B - k^2(\mathbf{v},\nabla p')_B = 0 \quad\forall(\mathbf{w}',\mathbf{v}',p')\in\boldsymbol{\mathcal{H}}_0.
\end{align*}
We remark that the last two terms on the left-hand side of this equation vanish by definition of $\boldsymbol{\mathcal{H}}_0$, and as a result we exclude them from this point onward. If we choose $(\mathbf{w}',\mathbf{v}',p') = (\mathbf{w},\mathbf{v},p)$ and take the imaginary part of this equation, then we see that
\begin{equation*}
k^2\Im(\eta)(\mathbf{v},\mathbf{v})_B = k^2(\Im(\epsilon)\mathbf{w},\mathbf{w})_B.
\end{equation*}
We observe that every eigenvalue $\eta$ must have nonnegative imaginary part, and it follows that $A_{-i\tau}^{\mathcal{T}}$ is injective whenever $\tau>0$. Thus, by Theorem \ref{theorem_Fredholm} and the analytic Fredholm theorem we conclude that the set of modified electromagnetic transmission eigenvalues is discrete without finite accumulation point. We summarize these results in the following theorem.


\begin{theorem} \label{theorem_discrete}
If $\gamma\neq1$, then the set of modified electromagnetic transmission eigenvalues is discrete in the complex plane, and, if they exist, each eigenvalue has nonnegative imaginary part.
\end{theorem}
We define the space
\begin{equation*}
\Hdivnormzero{B} := \{\mathbf{u}\in\mathbf{L}^2(B) \mid \div\mathbf{u} = 0 \text{ in } B, \; \un\cdot\mathbf{u} = 0 \text{ on } \partial B\}.
\end{equation*}
For a suitable choice of $z\in\mathbb{R}$, we next consider the operator $\boldPsi_z^{(\epsilon)}:\Hdivnormzero{B}\to\Hdivnormzero{B}$ defined by $\boldPsi_z^{(\epsilon)}\mathbf{g} := \mathbf{v}$, where $(\mathbf{w},\mathbf{v},p)\in\boldsymbol{\mathcal{H}}_0$ satisfies 
\begin{align}
\begin{split} \label{source1}
&(\curl\mathbf{w},\curl\mathbf{w}')_B - \gamma^{-1}(\curl\mathbf{v},\curl\mathbf{v}')_B \\
&\hspace{7em} - k^2(\epsilon\mathbf{w},\mathbf{w}')_B + k^2z(\mathbf{v},\mathbf{v}')_B = -k^2(\mathbf{g},\mathbf{v}')_B \quad\forall(\mathbf{w}',\mathbf{v}',p')\in\boldsymbol{\mathcal{H}}_0, 
\end{split}
\end{align}
which is an equivalent variational formulation of \eqref{nmit} with $\mathbf{f} = \mathbf{0}$, $\boldsymbol{\xi} = \mathbf{0}$, $\mathbf{h} = \mathbf{0}$, and $k^2\mathbf{g}$ in place of $\mathbf{g}$ for convenience. We note that the ability to choose $z\in\mathbb{R}$ such that \eqref{source1} is well-posed follows from Theorem \ref{theorem_discrete}. We remark that the scalar field $p$ no longer appears in the equation, but it is still determined uniquely by the properties of the space $\boldsymbol{\mathcal{H}}_0$. We first establish the following results for the operator $\boldPsi_z^{(\epsilon)}$.

\begin{prop} \label{prop_Psi}
The operator $\boldPsi_z^{(\epsilon)}$ is injective. Moreover, a given $\eta\in\mathbb{C}$ is a modified electromagnetic transmission eigenvalue if and only if $(\eta-z)^{-1}$ is an eigenvalue of $\boldPsi_z^{(\epsilon)}$.
\end{prop}

\begin{proof}
We suppose that $\boldPsi_z^{(\epsilon)}\mathbf{g} = \mathbf{0}$ for some $\mathbf{g}\in\Hdivnormzero{B}$, and from \eqref{source1} we have
\begin{equation}
(\curl\mathbf{w},\curl\mathbf{w}')_B - k^2(\epsilon\mathbf{w},\mathbf{w}')_B  = -k^2(\mathbf{g},\mathbf{v}')_B \quad\forall(\mathbf{w}',\mathbf{v}',p')\in\boldsymbol{\mathcal{H}}_0. \label{inj1}
\end{equation}
In particular, by choosing $\mathbf{w}'\in\Hcurltr{B}$ and $\mathbf{v}' = \mathbf{0}$ we obtain
\begin{equation*}
(\curl\mathbf{w},\curl\mathbf{w}')_B - k^2(\epsilon\mathbf{w},\mathbf{w}')_B  = 0 \quad\forall\mathbf{w}'\in\Hcurltr{B}.
\end{equation*}
We see that $\mathbf{w}$ satisfies \eqref{dir} with $\mathbf{f} = \mathbf{0}$ and $\boldsymbol{\xi} = \mathbf{0}$, and consequently by Assumption \ref{assumption_dirichlet} we conclude that $\mathbf{w} = \mathbf{0}$. From \eqref{inj1} we immediately have $\mathbf{g} = \mathbf{0}$, and it follows that $\boldPsi_z^{(\epsilon)}$ is injective. \par

For the second assertion, we suppose that $\boldPsi_z^{(\epsilon)} \mathbf{g} = \lambda\mathbf{g}$ for some $\lambda\in\mathbb{C}$ and nonzero $\mathbf{g}\in\Hdivnormzero{B}$. We note that $\lambda\neq0$ since $\boldPsi_z^{(\epsilon)}$ is injective. We may write $\mathbf{g} = \lambda^{-1}\mathbf{v}$, and from \eqref{source1} we have
\begin{equation*}
(\curl\mathbf{w},\curl\mathbf{w}')_B - \gamma^{-1}(\curl\mathbf{v},\curl\mathbf{v}')_B - k^2(\epsilon\mathbf{w},\mathbf{w}')_B + k^2z(\mathbf{v},\mathbf{v}')_B = -k^2(\lambda^{-1}\mathbf{v},\mathbf{v}')_B \quad\forall(\mathbf{w}',\mathbf{v}',p')\in\boldsymbol{\mathcal{H}}_0. 
\end{equation*}
By rearranging this expression we obtain
\begin{equation*}
(\curl\mathbf{w},\curl\mathbf{w}')_B - \gamma^{-1}(\curl\mathbf{v},\curl\mathbf{v}')_B - k^2(\epsilon\mathbf{w},\mathbf{w}')_B + k^2(z+\lambda^{-1})(\mathbf{v},\mathbf{v}')_B = 0 \quad\forall(\mathbf{w}',\mathbf{v}',p')\in\boldsymbol{\mathcal{H}}_0.
\end{equation*}
Since $\mathbf{v}\neq0$ by injectivity of $\boldPsi_z^{(\epsilon)}$, it follows that $\eta = z+\lambda^{-1}$ is a modified electromagnetic transmission eigenvalue. Noting that $\lambda = (\eta-z)^{-1}$ and following these steps in reverse order provides the converse result. \proofend
\end{proof}

As a result of Proposition \ref{prop_Psi} we may establish properties of modified electromagnetic transmission eigenvalues by studying the spectrum of $\boldPsi_z^{(\epsilon)}$. In the following lemma we prove that our modification of \eqref{mp} indeed results in a compact solution operator.

\begin{lemma} \label{lemma_Psi_compact}

The operator $\boldPsi_z^{(\epsilon)}:\Hdivnormzero{B}\to\Hdivnormzero{B}$ is compact.

\end{lemma}

\begin{proof}

We suppose that $\{\mathbf{g}_n\}$ is a bounded sequence in $\Hdivnormzero{B}$. The estimate \eqref{fredholm_est} combined with Theorem \ref{theorem_costabel} implies that
\begin{equation*}
\norm{\boldPsi_z^{(\epsilon)}\mathbf{g}}_{\mathbf{H}^{1/2}(B)} \le C\norm{\mathbf{g}}_B
\end{equation*}
for a constant $C>0$ independent of $\mathbf{g}$, and it follows that the sequence $\{\boldPsi_z^{(\epsilon)} \mathbf{g}_n\}$ is bounded in $\mathbf{H}^{1/2}(B)$. From the compact embedding of $\mathbf{H}^{1/2}(B)$ into $\mathbf{L}^2(B)$ we obtain a subsequence of $\{\boldPsi_z^{(\epsilon)} \mathbf{g}_n\}$ that converges to some $\mathbf{v}_0$ in $\mathbf{L}^2(B)$. Continuity of the divergence and normal trace operators yield $\mathbf{v}_0\in\Hdivnormzero{B}$ with convergence in this space as well, and we conclude that $\boldPsi_z^{(\epsilon)}$ is compact. \proofend
\end{proof}

With the result of this lemma, the spectral theorem for compact operators immediately provides another proof that the set of modified transmission eigenvalues is discrete without finite accumulation point, and next we see that if $\epsilon$ is real-valued it follows that eigenvalues must exist as well. Indeed, if we suppose that $\epsilon$ is real-valued and we let $(\mathbf{w}_j,\mathbf{v}_j,p_j)$ satisfy \eqref{source1} for $\mathbf{g} = \mathbf{g}_j\in\Hdivnormzero{B}$, $j = 1,2$, then we see that
\begin{align*}
k^2(\mathbf{g}_1,\boldPsi_z^{(\epsilon)}\mathbf{g}_2)_B &= k^2(\mathbf{g}_1,\mathbf{v}_2)_B \\
&= (\curl\mathbf{w}_1,\curl\mathbf{w}_2)_B - \gamma^{-1}(\curl\mathbf{v}_1,\curl\mathbf{v}_2)_B - k^2(\epsilon\mathbf{w}_1,\mathbf{w}_2)_B + k^2z(\mathbf{v}_1,\mathbf{v}_2)_B \\
&= \overline{(\curl\mathbf{w}_2,\curl\mathbf{w}_1)_B} - \gamma^{-1}\overline{(\curl\mathbf{v}_2,\curl\mathbf{v}_1)_B} - k^2\overline{(\epsilon\mathbf{w}_2,\mathbf{w}_1)_B} + k^2z\overline{(\mathbf{v}_2,\mathbf{v}_1)_B} \\
&= k^2\overline{(\mathbf{g}_2,\mathbf{v}_1)_B} \\
&= k^2(\boldPsi_z^{(\epsilon)}\mathbf{g}_1,\mathbf{g}_2)_B.
\end{align*}
Thus, the operator $\boldPsi_z^{(\epsilon)}$ is self-adjoint, and we summarize the immediate consequences of the spectral theorem for compact self-adjoint operators in the following theorem.

\begin{theorem} \label{theorem_eigs}

If $\gamma\neq1$ and $\epsilon$ is real-valued, then all of the modified electromagnetic transmission eigenvalues are real and infinitely many exist.

\end{theorem}

Unfortunately, if $\epsilon$ has a nonzero imaginary part, then we have no general result on existence of eigenvalues as this operator is no longer self-adjoint. However, a limited existence result applicable when $\epsilon$ has sufficiently small imaginary part will be presented in a forthcoming manuscript (see \cite{cakoni_colton_haddar} for a similar result for standard transmission eigenvalues).

\section{Determination of modified transmission eigenvalues from electric far field data} \label{sec_determine}

In this section we establish that modified electromagnetic transmission eigenvalues may be computed from electric far field data using the linear sampling method. We note that we correct some errors in a similar analysis performed in \cite{camano_lackner_monk} for electromagnetic Stekloff eigenvalues. We begin by recalling that an electric dipole with polarization $\mathbf{q}$ is defined by
\begin{align}
\begin{split}
\mathbf{E}_e(\mathbf{x},\mathbf{z},\mathbf{q}) &:= \frac{i}{k}\textbf{curl}_\mathbf{x}\,\textbf{curl}_\mathbf{x}\, \mathbf{q}\Phi(\mathbf{x},\mathbf{z}), \\
\mathbf{H}_e(\mathbf{x},\mathbf{z},\mathbf{q}) &:= \textbf{curl}_\mathbf{x}\, \mathbf{q}\Phi(\mathbf{x},\mathbf{z}).
\end{split}
\end{align}
The electric field $\mathbf{E}_e(\cdot,\mathbf{z},\mathbf{q})$ is a radiating solution to Maxwell's equations outside of a neighborhood of $\mathbf{z}$ with corresponding far field pattern
\begin{equation*}
\mathbf{E}_{e,\infty}(\hat{\mathbf{x}},\mathbf{z},\mathbf{q}) := \frac{ik}{4\pi}(\hat{\mathbf{x}}\times\mathbf{q})\times\hat{\mathbf{x}}\, e^{-ik\hat{\mathbf{x}}\cdot\mathbf{z}}.
\end{equation*}
We investigate the \emph{modified far field equation} for $\mathbf{z}\in B$, which is to find $\mathbf{g}\in\mathbf{L}_t^2(\mathbb{S}^2)$ satisfying
\begin{equation}
(\mathcal{F}\mathbf{g})(\hat{\mathbf{x}}) = \mathbf{E}_{e,\infty}(\hat{\mathbf{x}},\mathbf{z},\mathbf{q}). \label{mffe}
\end{equation}
If $\mathbf{g}_\mathbf{z}\in\mathbf{L}_t^2(\mathbb{S}^2)$ satisfies \eqref{mffe}, then we see from Rellich's lemma (cf. \cite{colton_kress}) that
\begin{equation*}
\mathbf{w}_\mathbf{z}(\mathbf{x}) - \mathbf{v}_\mathbf{z}(\mathbf{x}) = \mathbf{w}_\mathbf{z}^s(\mathbf{x}) - \mathbf{v}_\mathbf{z}^s(\mathbf{x}) = \mathbf{E}_e(\mathbf{x},\mathbf{z},\mathbf{q}), \quad \mathbf{x}\in\mathbb{R}^3\setminus\overline{B},
\end{equation*}
where $(\mathbf{w}_\mathbf{z},\mathbf{w}_\mathbf{z}^s)$ and $(\mathbf{v}_\mathbf{z},\mathbf{v}_\mathbf{z}^s,p_\mathbf{z})$ satisfy \eqref{sc} and \eqref{aux}, respectively, with incident field given by the Herglotz wave function $\mathbf{E}^i = \mathbf{v}_{\mathbf{g}_\mathbf{z}}^i$ that we defined in \eqref{herglotz}. It follows that $(\mathbf{w}_\mathbf{z},\mathbf{v}_\mathbf{z},p_\mathbf{z})$ satisfies the modified interior transmission problem
\begin{subequations} \label{zmit}
\begin{align}
\curl\curl\mathbf{w}_\mathbf{z} - k^2\epsilon\mathbf{w}_\mathbf{z} &= \mathbf{0} \text{ in } B, \label{zmit1} \\
\curl\gamma^{-1}\curl\mathbf{v}_\mathbf{z} - k^2\eta\mathbf{v}_\mathbf{z} + k^2\nabla p_\mathbf{z} &= \mathbf{0} \text{ in } B, \label{zmit2} \\
\div\mathbf{v}_\mathbf{z} &= 0 \text{ in } B, \label{zmit3} \\
\un\cdot\mathbf{v}_\mathbf{z} &= 0 \text{ on } \partial B, \label{zmit4} \\
\un\times(\mathbf{w}_\mathbf{z}-\mathbf{v}_\mathbf{z}) &= \un\times\mathbf{E}_e(\cdot,z,\mathbf{q}) \text{ on } \partial B, \label{zmit5} \\
\un\times\left(\curl\mathbf{w}_\mathbf{z} - \gamma^{-1}\curl\mathbf{v}_\mathbf{z}\right) &= \un\times\curl\mathbf{E}_e(\cdot,\mathbf{z},\mathbf{q}) \text{ on } \partial B, \label{zmit6}
\end{align}
\end{subequations}
and that these fields admit the decomposition
\begin{equation}
\mathbf{w}_\mathbf{z} = \mathbf{v}_{\mathbf{g}_\mathbf{z}}^i + \mathbf{w}_\mathbf{z}^s, \quad \mathbf{v}_\mathbf{z} = \mathbf{v}_{\mathbf{g}_\mathbf{z}}^i + \mathbf{v}_\mathbf{z}^s \text{ in } \mathbb{R}^3. \label{decomp}
\end{equation}
Unfortunately, the solution of \eqref{zmit} cannot in general be decomposed as in \eqref{decomp}, but in the following lemma we show that these fields may be decomposed as the sum of an incident field and a radiating field such that the incident fields coincide.

\begin{lemma} \label{lemma_decomp}

If $\eta$ is not a modified transmission eigenvalue, then \eqref{zmit} has a unique solution $(\mathbf{w}_\mathbf{z},\mathbf{v}_\mathbf{z},p_\mathbf{z})\in\Hcurl{B}\times\Hcurl{B}\times H_*^1(B)$ and the fields $\mathbf{w}_\mathbf{z}$, $\mathbf{v}_\mathbf{z}$ may be decomposed as
\begin{equation*}
\mathbf{w}_\mathbf{z} = \mathbf{u}_\mathbf{z}^i + \mathbf{w}_\mathbf{z}^s, \quad \mathbf{v}_\mathbf{z} = \mathbf{u}_\mathbf{z}^i + \mathbf{v}_\mathbf{z}^s,
\end{equation*}
where $\mathbf{u}_\mathbf{z}^i\in\Hcurl{B}$ satisfies the free-space Maxwell's equations in $B$ and $\mathbf{w}_\mathbf{z}^s,\mathbf{v}_\mathbf{z}^s\in\Hcurlloc{\mathbb{R}^3}$ are radiating solutions of the free-space Maxwell's equations outside of $B$.

\end{lemma}

\begin{proof}

Since $\eta$ is not a modified transmission eigenvalue, it follows from Theorem \ref{theorem_Fredholm} that \eqref{zmit} possesses a unique solution which depends continuously on the data. In order to arrive at the desired decompositions of $\mathbf{w}_\mathbf{z}$ and $\mathbf{v}_\mathbf{z}$, we apply the Stratton-Chu representation formula (cf. \cite{colton_kress}). First, we define the incident field
\begin{align*}
\mathbf{w}_\mathbf{z}^i(\mathbf{x}) &:= -\curl\int_{\partial B} \un(\mathbf{y})\times\mathbf{w}_\mathbf{z}(\mathbf{y}) \Phi(\mathbf{x},\mathbf{y}) ds(\mathbf{y}) + \grad\int_{\partial B} \un(\mathbf{y})\cdot[\epsilon(\mathbf{y})\mathbf{w}_\mathbf{z}(\mathbf{y})] \Phi(\mathbf{x},\mathbf{y}) ds(\mathbf{y}) \\
&\hspace{10em} - \int_{\partial B} \un(\mathbf{y})\times\curl\mathbf{w}_\mathbf{z}(\mathbf{y}) \Phi(\mathbf{x},\mathbf{y}) ds(\mathbf{y}), \quad \mathbf{x}\in B,
\end{align*}
and the scattered field
\begin{align*}
\mathbf{w}_\mathbf{z}^s(\mathbf{x}) &:= -\grad\int_{\partial B}\un(\mathbf{y})\cdot[(\epsilon(\mathbf{y})-1)\mathbf{w}_\mathbf{z}(\mathbf{y})] \Phi(\mathbf{x},\mathbf{y}) ds(\mathbf{y}) + \grad\int_B \div[(\epsilon(\mathbf{y})-1)\mathbf{w}_\mathbf{z}(\mathbf{y})] \Phi(\mathbf{x},\mathbf{y}) d\mathbf{y} \\
&\hspace{10em} + k^2\int_B [\epsilon(\mathbf{y})-1]\mathbf{w}_\mathbf{z}(\mathbf{y}) \Phi(\mathbf{x},\mathbf{y}) d\mathbf{y}, \quad \mathbf{x}\in\mathbb{R}^3,
\end{align*}
where we have used the fact that $\div(\epsilon\mathbf{w}_\mathbf{z}) = 0$ in $B$. By similar reasoning to \cite[p. 193]{colton_kress}, the incident field may be written as
\begin{equation*}
\mathbf{w}_\mathbf{z}^i(\mathbf{x}) := -\curl\int_{\partial B} \un(\mathbf{y})\times\mathbf{w}_\mathbf{z}(\mathbf{y}) \Phi(\mathbf{x},\mathbf{y}) ds(\mathbf{y}) - \frac{1}{k^2}\curl\curl\int_{\partial B} \un(\mathbf{y})\times\curl\mathbf{w}_z(\mathbf{y}) \Phi(\mathbf{x},\mathbf{y}) ds(\mathbf{y}), \quad \mathbf{x}\in B,
\end{equation*}
and consequently we see that $\mathbf{w}_\mathbf{z}^i$ satisfies the free-space Maxwell's equations in $B$. Moreover, an application of the identity
\begin{equation*}
\curl\curl = -\boldsymbol{\Delta} + \grad\div
\end{equation*}
and the divergence theorem implies that $\mathbf{w}_\mathbf{z}^s$ is a radiating solution of the free-space Maxwell's equations in $\mathbb{R}^3\setminus\overline{B}$. A similar decomposition appeared in the proof of Lemma 4.1 in \cite{camano_lackner_monk}, but each of the incident and scattered fields is missing a necessary term which is corrected in the above decomposition. Similarly, we define the incident field
\begin{align*}
\mathbf{v}_\mathbf{z}^i(\mathbf{x}) &:= -\curl\int_{\partial B} \un(\mathbf{y})\times\mathbf{v}_\mathbf{z}(\mathbf{y}) \Phi(\mathbf{x},\mathbf{y}) ds(\mathbf{y}) - \grad\int_{\partial B} \un(\mathbf{y})\cdot\nabla p_\mathbf{z}(\mathbf{y}) \Phi(\mathbf{x},\mathbf{y}) ds(\mathbf{y}) \\
&\hspace{10em} - \int_{\partial B} \un(\mathbf{y})\times\curl\mathbf{v}_\mathbf{z}(\mathbf{y}) \Phi(\mathbf{x},\mathbf{y}) ds(\mathbf{y}), \quad \mathbf{x}\in B,
\end{align*}
and the scattered field
\begin{align*}
\mathbf{v}_\mathbf{z}^s(\mathbf{x}) &:= \grad\int_{\partial B}\un(\mathbf{y})\cdot\nabla p_\mathbf{z}(\mathbf{y}) \Phi(\mathbf{x},\mathbf{y}) ds(\mathbf{y}) + \curl\int_B \left(1 - \gamma^{-1}\right)\curl\mathbf{v}_\mathbf{z}(\mathbf{y}) \Phi(\mathbf{x},\mathbf{y}) d\mathbf{y} \\
&\hspace{10em} + k^2\int_B [(\eta - 1)\mathbf{v}_\mathbf{z}(\mathbf{y}) - \nabla p_\mathbf{z}(\mathbf{y})] \Phi(\mathbf{x},\mathbf{y}) d\mathbf{y}, \quad \mathbf{x}\in\mathbb{R}^3,
\end{align*}
where we have used the fact that $\div\mathbf{v}_\mathbf{z} = 0$ in $B$ and $\un\cdot\mathbf{v}_\mathbf{z} = 0$ on $\partial B$. By the same arguments we see that $\mathbf{v}_\mathbf{z}^i$ and $\mathbf{v}_\mathbf{z}^s$ satisfy the free-space Maxwell's equations in $B$ and $\mathbb{R}^3\setminus\overline{B}$, respectively, and the scattered field satisfies the radiation condition. Since $\mathbf{w}_\mathbf{z} = \mathbf{w}_\mathbf{z}^i + \mathbf{w}_\mathbf{z}^s$ and $\mathbf{v}_\mathbf{z} = \mathbf{v}_\mathbf{z}^i + \mathbf{v}_\mathbf{z}^s$ by the Stratton-Chu formula, it remains to prove that $\mathbf{w}_\mathbf{z}^i = \mathbf{v}_\mathbf{z}^i$. By the boundary conditions \eqref{zmit} and the relation
\begin{equation*}
\un\cdot(\epsilon\mathbf{w}_\mathbf{z}) + \normal{p_\mathbf{z}} = \un\cdot\mathbf{E}_e(\cdot,\mathbf{z},\mathbf{q}) \text{ on } \partial B,
\end{equation*}
we see that
\begin{align*}
\mathbf{w}_\mathbf{z}^i(\mathbf{x}) - \mathbf{v}_\mathbf{z}^i(\mathbf{x})  &:= -\curl\int_{\partial B} \un(\mathbf{y})\times\mathbf{E}_e(\mathbf{y},\mathbf{z},\mathbf{q}) \Phi(\mathbf{x},\mathbf{y}) ds(\mathbf{y}) + \grad\int_{\partial B} \un(\mathbf{y})\cdot\mathbf{E}_e(\mathbf{y},\mathbf{z},\mathbf{q}) \Phi(\mathbf{x},\mathbf{y}) ds(\mathbf{y}) \\
&\hspace{10em} - \int_{\partial B} \un(\mathbf{y})\times\curl\mathbf{E}_e(\mathbf{y},\mathbf{z},\mathbf{q}) \Phi(\mathbf{x},\mathbf{y}) ds(\mathbf{y}), \quad \mathbf{x}\in B.
\end{align*}
As we did for the incident field $\mathbf{w}_\mathbf{z}^i$, we may write this difference as
\begin{align*}
\mathbf{w}_\mathbf{z}^i(\mathbf{x}) - \mathbf{v}_\mathbf{z}^i(\mathbf{x})  &:= -\curl\int_{\partial B} \un(\mathbf{y})\times\mathbf{E}_e(\mathbf{y},\mathbf{z},\mathbf{q}) \Phi(\mathbf{x},\mathbf{y}) ds(\mathbf{y}) \\
&\quad\quad\quad\quad\quad- \frac{1}{k^2}\curl\curl\int_{\partial B} \un(\mathbf{y})\times\curl\mathbf{E}_e(\mathbf{y},\mathbf{z},\mathbf{q}) \Phi(\mathbf{x},\mathbf{y}) ds(\mathbf{y}), \quad \mathbf{x}\in B.
\end{align*}
We now consider a ball $B_R$ of radius $R>0$ sufficiently large that $\overline{B}\subset B_R$, and we define the domain $S_R := B_R\setminus\overline{B}$ with boundary $\partial S_R = \partial B_R\cup\partial B$, where the unit normal $\un$ on $\partial B_R$ is directed into the exterior of $B_R$ and the unit normal $\un$ on $\partial B$ is directed into the interior of $S_R$. Since any $\mathbf{x}\in B$ lies outside of $S_R$ we see that
\begin{equation*}
-\curl\int_{\partial S_R} \un(\mathbf{y})\times\mathbf{E}_e(\mathbf{y},\mathbf{z},\mathbf{q}) \Phi(\mathbf{x},\mathbf{y}) ds(\mathbf{y}) - \frac{1}{k^2}\curl\curl\int_{\partial S_R} \un(\mathbf{y})\times\curl\mathbf{E}_e(\mathbf{y},\mathbf{z},\mathbf{q}) \Phi(\mathbf{x},\mathbf{y}) ds(\mathbf{y}) = 0,
\end{equation*}
and consequently we obtain the representation
\begin{align*}
\mathbf{w}_\mathbf{z}^i(\mathbf{x}) - \mathbf{v}_\mathbf{z}^i(\mathbf{x})  &:= -\curl\int_{\partial B_R} \un(\mathbf{y})\times\mathbf{E}_e(\mathbf{y},\mathbf{z},\mathbf{q}) \Phi(\mathbf{x},\mathbf{y}) ds(\mathbf{y}) \\
&\quad\quad\quad\quad\quad- \frac{1}{k^2}\curl\curl\int_{\partial B_R} \un(\mathbf{y})\times\curl\mathbf{E}_e(\mathbf{y},\mathbf{z},\mathbf{q}) \Phi(\mathbf{x},\mathbf{y}) ds(\mathbf{y}), \quad \mathbf{x}\in B,
\end{align*}
for any sufficiently large $R>0$. From the Silver-M{\"u}ller radiation condition satisfied by $\mathbf{E}_e(\cdot,\mathbf{z},q)$ and $\mathbf{H}_e(\cdot,\mathbf{z},q)$ it follows as in \cite[Theorem 6.7]{colton_kress} that this expression vanishes as $R\to\infty$ for any $\mathbf{x}\in B$. We conclude that $\mathbf{w}_\mathbf{z}^i = \mathbf{v}_\mathbf{z}^i$ and hence we may denote both incident fields by $\mathbf{u}_\mathbf{z}^i$. \proofend

\end{proof}

We now factorize the modified far field operator $\boldsymbol{\mathcal{F}}$. We begin by defining the space of generalized incident fields as
\begin{equation*}
\Hcurlinc{B} := \{\mathbf{u}^i\in\Hcurl{B} \mid \curl\curl\mathbf{u}^i - k^2\mathbf{u}^i = \mathbf{0} \text{ in } B\},
\end{equation*}
and we define the Herglotz operator $\mathbf{H}:\mathbf{L}_t^2(\mathbb{S}^2)\to\Hcurlinc{B}$ by $\mathbf{H}\mathbf{g} := \mathbf{v}_\mathbf{g}^i$. We also define two solution operators as follows. We define $\mathbf{G}:\Hcurlinc{B}\to\mathbf{L}_t^2(\mathbb{S}^2)$ by $\mathbf{G}\mathbf{u}^i := \mathbf{w}_\infty^*$, where $\mathbf{w}_\infty^*$ is the far field pattern corresponding to the unique radiating solution $\mathbf{w}^*\in\Hcurlloc{\mathbb{R}^3}$ of
\begin{equation}
\curl\curl\mathbf{w}^* - k^2\epsilon\mathbf{w}^* = k^2(1-\epsilon)\mathbf{u}^i \text{ in } \mathbb{R}^3. \label{wstar}
\end{equation}
We define $\mathbf{G}_0:\Hcurlinc{B}\to\mathbf{L}_t^2(\mathbb{S}^2)$ by $\mathbf{G}_0\mathbf{u}^i := \mathbf{v}_\infty^*$, where $\mathbf{v}_\infty^*$ is the far field pattern corresponding to the unique solution $(\mathbf{v}^*,p^*)\in\Hcurlloc{\mathbb{R}^3}\times H_*^1(B)$ of
\begin{subequations} \label{vstar}
\begin{align}
\curl\tilde{\gamma}^{-1}\curl\mathbf{v}^* - k^2\tilde{\eta}\mathbf{w}^* + k^2\chi_B\nabla p^* &= \curl\left(1-\tilde{\gamma}^{-1}\right)\curl\mathbf{u}^i + k^2(1-\tilde{\eta})\mathbf{u}^i \text{ in } \mathbb{R}^3, \label{vstar1} \\
\div\mathbf{v}^* &= 0 \text{ in } B, \label{vstar2} \\
\un\cdot\mathbf{v}^* &= -\un\cdot\mathbf{u}^i \text{ on } \partial B. \label{vstar3}
\end{align}
\end{subequations}
We denote by $\chi_B$ the characteristic function for the domain $B$, and for any constant $a$ we denote by $\tilde{a}$ the function which has the constant values of $a$ in $B$ and $1$ elsewhere. We see that $\mathbf{F} = \mathbf{G}\mathbf{H}$ and $\mathbf{F}_0 = \mathbf{G}_0 \mathbf{H}$, and consequently we obtain the factorization $\boldsymbol{\mathcal{F}} = \boldsymbol{\mathcal{G}} \mathbf{H}$, where the operator $\boldsymbol{\mathcal{G}}:\Hcurlinc{B}\to\mathbf{L}_t^2(\mathbb{S}^2)$ is defined by $\boldsymbol{\mathcal{G}} := \mathbf{G} - \mathbf{G}_0$ and is compact. With this factorization in hand, we provide a characterization of the modified transmission eigenvalues in terms of $\boldsymbol{\mathcal{F}}$ in the following two theorems.

\begin{theorem} \label{theorem_noteig}

Let $\mathbf{z}\in B$. If $\eta$ is not a modified electromagnetic transmission eigenvalue, then for every $\delta>0$ there exists $\mathbf{g}_z^\delta\in\mathbf{L}_t^2(\mathbb{S}^2)$ satisfying
\begin{equation}
\lim_{\delta\to0} \norm{\boldsymbol{\mathcal{F}}\mathbf{g}_\mathbf{z}^\delta - \mathbf{E}_{e,\infty}(\cdot,\mathbf{z},\mathbf{q})}_{\mathbf{L}_t^2(\mathbb{S}^2)} = 0 \label{F_limit}
\end{equation}
such that the sequence $\left\{\mathbf{v}_{\mathbf{g}_\mathbf{z}^\delta}^i\right\}_{\delta>0}$ of Herglotz wave functions is convergent and hence bounded in $\Hcurl{B}$ as $\delta\to0$.

\end{theorem}

\begin{proof}

By the assumption that $\eta$ is not a modified electromagnetic transmission eigenvalue, it follows from Lemma \ref{lemma_decomp} that there exists a unique solution $(\mathbf{w}_\mathbf{z},\mathbf{v}_\mathbf{z},p_\mathbf{z})$ of \eqref{zmit} with the decomposition
\begin{equation*}
\mathbf{w}_\mathbf{z} = \mathbf{u}_\mathbf{z}^i + \mathbf{w}_\mathbf{z}^s, \quad \mathbf{v}_\mathbf{z} = \mathbf{u}_\mathbf{z}^i + \mathbf{v}_\mathbf{z}^s \text{ in } B,
\end{equation*}
where $\mathbf{u}_\mathbf{z}^i\in\Hcurlinc{B}$ and both of the scattered fields $\mathbf{w}_\mathbf{z}^s,\mathbf{v}_\mathbf{z}^s\in\Hcurlloc{\mathbb{R}^3}$ satisfy the free-space Maxwell's equations in $\mathbb{R}^3\setminus\overline{B}$ along with the Silver-M{\"u}ller radiation condition. The definition of $\boldsymbol{\mathcal{G}}$ implies that $\boldsymbol{\mathcal{G}}\mathbf{u}_\mathbf{z}^i = \mathbf{E}_{e,\infty}(\cdot,\mathbf{z},\mathbf{q})$. We note that the range of the Herglotz operator $\mathbf{H}$ is dense in $\Hcurlinc{B}$ (cf. \cite{coltonkress2001}), and consequently for any $\delta>0$ we may choose $\mathbf{g}_\mathbf{z}^\delta\in\mathbf{L}_t^2(\mathbb{S}^2)$ such that
\begin{equation}
\norm{\mathbf{v}_{\mathbf{g}_\mathbf{z}^\delta}^i - \mathbf{u}_\mathbf{z}^i}_{\Hcurl{B}} < \frac{\delta}{\norm{\boldsymbol{\mathcal{G}}}}. \label{Herglotz_limit}
\end{equation}
From the observation that $\boldsymbol{\mathcal{F}}\mathbf{g}_\mathbf{z}^\delta = \boldsymbol{\mathcal{G}}\mathbf{v}_{\mathbf{g}_\mathbf{z}^\delta}^i$ we obtain
\begin{align*}
\norm{\boldsymbol{\mathcal{F}}\mathbf{g}_\mathbf{z}^\delta - \mathbf{E}_{e,\infty}(\cdot,\mathbf{z},\mathbf{q})}_{\mathbf{L}_t^2(\mathbb{S}^2)} &\le \norm{\boldsymbol{\mathcal{G}}}\norm{\mathbf{v}_{\mathbf{g}_\mathbf{z}^\delta}^i - \mathbf{u}_\mathbf{z}^i}_{\Hcurl{B}} < \delta,
\end{align*}
which clearly implies \eqref{F_limit}. Moreover, we see from \eqref{Herglotz_limit} that $\mathbf{v}_{\mathbf{g}_\mathbf{z}^\delta}^i \to \mathbf{u}_\mathbf{z}^i$ in $\Hcurl{B}$ as $\delta\to0$. \proofend

\end{proof}

Before we prove the next theorem, we recall a result from \cite{camano_lackner_monk}. We remark that the original statement of the result in Lemma 4.3 of that work contains an error in the first equation. In particular, the term $\frac{1}{k^2}$ is missing, and we provide the corrected version here.

\begin{lemma} \label{lemma_rep}

For all $\mathbf{z}\in B$, $\mathbf{q}\in\mathbb{R}^3$, and sufficiently regular functions $\mathbf{u}$, we have the identities
\begin{align*}
&\int_{\partial B} \curl\mathbf{u}(\mathbf{x})\cdot\biggr(\un(\mathbf{x})\times\mathbf{E}_e(\mathbf{x},\mathbf{z},\mathbf{q}) \biggr) ds(\mathbf{x}) \\
&\hspace{5em} = ik\mathbf{q}\cdot \left( \frac{1}{k^2}\mathbf{grad}_\mathbf{z}\textnormal{div}_\mathbf{z}\int_{\partial B} \curl\mathbf{u}(\mathbf{x})\times\un(\mathbf{x}) \Phi(\mathbf{x},\mathbf{z}) ds(\mathbf{x}) \right. - \left. \int_{\partial B} \curl\mathbf{u}(\mathbf{x})\times\un(\mathbf{x}) \Phi(\mathbf{x},\mathbf{z}) ds(\mathbf{x}) \right)
\end{align*}
and
\begin{equation*}
\int_{\partial B} \biggr( \un(\mathbf{x})\times\mathbf{u}(\mathbf{x}) \biggr)\cdot\mathbf{curl}_\mathbf{x}\mathbf{E}_e(\mathbf{x},\mathbf{z},\mathbf{q}) ds(\mathbf{x}) = ik\mathbf{q}\cdot\mathbf{curl}_\mathbf{z}\int_{\partial B} \un(\mathbf{x})\times\mathbf{u}(\mathbf{x}) \Phi(\mathbf{x},\mathbf{z}) ds(\mathbf{x}).
\end{equation*}

\end{lemma}

\begin{theorem} \label{theorem_eig}

If $\eta$ is a modified electromagnetic transmission eigenvalue and the sequence $\{\mathbf{g}_\mathbf{z}^\delta\}_{\delta>0}$ satisfies \eqref{F_limit} for a given $\mathbf{z}\in B$, then the sequence $\{\mathbf{v}_{\mathbf{g}_\mathbf{z}^\delta}^i\}$ cannot be bounded in $\Hcurl{B}$ as $\delta\to0$ for almost every $\mathbf{z}\in B_\rho$, where $B_\rho\subset B$ is an arbitrary ball of radius $\rho>0$.

\end{theorem}

\begin{proof}

We suppose to the contrary that for some ball $B_\rho\subset B$ and all $\mathbf{z}\in B_\rho$ the sequence $\left\{\mathbf{v}_{\mathbf{g}_\mathbf{z}^\delta}^i\right\}$ is bounded in $\Hcurl{B}$ as $\delta\to0$, which implies that, upon passing to a subsequence, $\left\{\mathbf{v}_{\mathbf{g}_\mathbf{z}^\delta}^i\right\}$ converges weakly to some $\mathbf{u}_\mathbf{z}^i$ in $\Hcurlinc{B}$. By compactness of $\boldsymbol{\mathcal{G}}$ we obtain
\begin{equation*}
\boldsymbol{\mathcal{G}}\mathbf{v}_{\mathbf{g}_\mathbf{z}^\delta}^i \to \boldsymbol{\mathcal{G}}\mathbf{u}_\mathbf{z}^i \text{ in } \mathbf{L}_t^2(\mathbb{S}^2) \text{ as } \delta\to0,
\end{equation*}
and it follows from \eqref{F_limit} that $\boldsymbol{\mathcal{G}}\mathbf{u}_\mathbf{z}^i = \mathbf{E}_{e,\infty}(\cdot,\mathbf{z},\mathbf{q})$. If we let $\mathbf{w}_\mathbf{z}^*\in\Hcurlloc{\mathbb{R}^3}$ and $(\mathbf{v}_\mathbf{z}^*,p_\mathbf{z}^*)\in\Hcurlloc{\mathbb{R}^3}\times H_*^1(B)$ be the unique solutions of \eqref{wstar} and \eqref{vstar}, respectively, for the incident field $\mathbf{u}_\mathbf{z}^i$, then we see by definition of $\boldsymbol{\mathcal{G}}$ that
\begin{equation*}
\mathbf{w}_{\mathbf{z},\infty}^* - \mathbf{v}_{\mathbf{z},\infty}^* = \mathbf{E}_{e,\infty}(\cdot,\mathbf{z},\mathbf{q}).
\end{equation*}
An application of Rellich's lemma implies that $(\mathbf{w}_\mathbf{z},\mathbf{v}_\mathbf{z},p_\mathbf{z}) := (\mathbf{w}_\mathbf{z}^*|_B + \mathbf{u}_\mathbf{z}^i,\mathbf{v}_\mathbf{z}^*|_B + \mathbf{u}_\mathbf{z}^i,p_\mathbf{z})$ satisfies \eqref{zmit}. Since $\eta$ is a modified electromagnetic transmission eigenvalue, there exists a nontrivial solution $(\mathbf{w}_\eta,\mathbf{v}_\eta,p_\eta)$ of \eqref{mit}. We may apply Green's second vector theorem as in \cite{colton_kress} along with some simple calculations to obtain
\begin{align*}
\int_{\partial B} \left( \un\times\mathbf{v}_\mathbf{z}\cdot\gamma^{-1}\curl\mathbf{v}_\eta - \un\times\mathbf{v}_\eta\cdot\gamma^{-1}\curl\mathbf{v}_\mathbf{z} \right) ds &= 0, \\
\int_{\partial B} \left( \un\times\mathbf{w}_\mathbf{z}\cdot\curl\mathbf{w}_\eta - \un\times\mathbf{w}_\eta\cdot\curl\mathbf{w}_\mathbf{z} \right) ds &= 0,
\end{align*}
and upon subtracting these equations and applying the boundary conditions \eqref{zmit5}--\eqref{zmit6} and \eqref{mit5}--\eqref{mit6} we see that
\begin{equation*}
\int_{\partial B} \biggr( \un\times\mathbf{E}_e(\cdot,\mathbf{z},\mathbf{q})\cdot\curl\mathbf{w}_\eta - \un\times\mathbf{w}_\eta\cdot\curl\mathbf{E}_e(\cdot,\mathbf{z},\mathbf{q}) \biggr) ds = 0.
\end{equation*}
We now invoke the identities of Lemma \ref{lemma_rep} to write this equation as
\begin{align}
\begin{split}
ik\mathbf{q}\cdot \Bigr( \frac{1}{k^2}\textbf{grad}_\mathbf{z}\textnormal{div}_\mathbf{z}&\int_{\partial B} \curl\mathbf{w}_\eta\times\un \Phi(\cdot,\mathbf{z}) ds + \int_{\partial B} \curl\mathbf{w}_\eta\times\un \Phi(\cdot,\mathbf{z}) ds \\
&\hspace{12em} - \textbf{curl}_\mathbf{z}\int_{\partial B} \un\times\mathbf{w}_\eta \Phi(\cdot,\mathbf{z}) ds \Bigr) = 0, \; \mathbf{z}\in B_\rho. \label{Lambda_normal}
\end{split}
\end{align}
We define the function
\begin{align*}
\boldsymbol{\Lambda}(\mathbf{z}) := \frac{1}{k^2}&\textbf{grad}_\mathbf{z}\textnormal{div}_\mathbf{z}\int_{\partial B} \curl\mathbf{w}_\eta\times\un \Phi(\cdot,\mathbf{z}) ds + \int_{\partial B} \curl\mathbf{w}_\eta\times\un \Phi(\cdot,\mathbf{z}) ds \\
&\hspace{12em} - \textbf{curl}_\mathbf{z}\int_{\partial B} \un\times\mathbf{w}_\eta \Phi(\cdot,\mathbf{z}) ds, \; \mathbf{z}\in\mathbb{R}^3\setminus\partial B,
\end{align*}
and we observe that $\boldsymbol{\Lambda}$ satisfies the free-space Maxwell's equations in both $\mathbb{R}^3\setminus\overline{B}$ and $B$ and that it satisfies the radiation condition. From \eqref{Lambda_normal} we see that $ik\mathbf{q}\cdot\boldsymbol{\Lambda}(\mathbf{z}) = 0$ for all $\mathbf{z}\in B_\rho$ and all $\mathbf{q}\in\mathbb{R}^3$, and the unique continuation principle implies that $\boldsymbol{\Lambda}(\mathbf{z}) = 0$ for all $\mathbf{z}\in B$. It follows that
\begin{equation*}
\un\times\boldsymbol{\Lambda}^- = \mathbf{0} \text{ and } \un\times\curl\boldsymbol{\Lambda}^- = \mathbf{0} \text{ on } \partial B,
\end{equation*}
where the superscripts $+$ and $-$ denote the trace on $\partial B$ from the exterior and interior of $B$, respectively. The jump relations of vector potentials (cf. \cite[Theorem 6.12]{colton_kress}) then imply that
\begin{equation*}
\un\times\boldsymbol{\Lambda}^+ = -\un\times\mathbf{w}_\eta \text{ and } \un\times\curl\boldsymbol{\Lambda}^+ = -\un\times\curl\mathbf{w}_\eta \text{ on } \partial B,
\end{equation*}
from which it follows that $\mathbf{E}^s := -\boldsymbol{\Lambda}\in\Hcurlloc{\mathbb{R}^3\setminus\overline{B}}$ and $\mathbf{E} := \mathbf{w}_\eta\in\Hcurl{B}$ satisfy \eqref{sc} with $B$ in place of $D$ and $\mathbf{E}^i = 0$. (Note that this problem is equivalent to \eqref{sc} upon redefining the total and scattered fields since $\epsilon=1$ in $B\setminus\overline{D}$.) Since this problem is well-posed we conclude that both $\boldsymbol{\Lambda}$ and $\mathbf{w}_\eta$ are identically zero, which by the boundary conditions \eqref{mit} implies that
\begin{equation*}
\un\times\mathbf{v}_\eta = \mathbf{0} \text{ and } \un\times\gamma^{-1}\curl\mathbf{v}_\eta = \mathbf{0} \text{ on } \partial B.
\end{equation*}
The same arguments used in the proof of uniqueness of the auxiliary problem \eqref{aux} yield $\mathbf{v}_\eta = \mathbf{0}$ and $p_\eta = 0$ in $B$, which contradicts the assumption that the solution $(\mathbf{w}_\eta,\mathbf{v}_\eta,p_\eta)$ of the homogeneous modified interior transmission problem was nonzero. \proofend

\end{proof}

\section{Numerical examples} \label{sec_numerics}

We begin this section with a brief explanation of how the results of Theorems \ref{theorem_noteig} and \ref{theorem_eig} allow us to detect modified electromagnetic transmission eigenvalues from electric far field data via the linear sampling method (LSM). We begin with the measured scattering data represented by the electric far field operator $\mathbf{F}$, and we select a rectangular region in the complex plane (or an interval on the real line if $\epsilon$ is real-valued) and construct a grid of values of $\eta$ in which to seek eigenvalues. For each such $\eta$, we compute the auxiliary far field operator $\mathbf{F}_0$ and construct the modified far field operator $\boldsymbol{\mathcal{F}}$. In practice these operators are discretized by computing electric far field patterns for various choices of incident direction $\mathbf{d}$ on the unit sphere. We use Tikhonov regularization to approximately solve a discretized version of the modified far field equation \eqref{mffe} for multiple choices of source points $\mathbf{z}$ and polarization vectors $\mathbf{q}$. In each case we compute the $\mathbf{L}_t^2(\mathbb{S}^2)$-norm of the solution $\mathbf{g}_{\eta,\mathbf{z},\mathbf{q}}$ (which serves as a proxy for the $\Hcurl{B}$-norm of $\mathbf{v}_{\mathbf{g}_{\eta,\mathbf{z},\mathbf{q}}}^i$), average it over all $\mathbf{z}$ and $\mathbf{q}$ to obtain a number $g_\eta$, and investigate the contour plot of the LSM indicator function $\eta\mapsto g_\eta$. Theorems \ref{theorem_noteig} and \ref{theorem_eig} suggest that the number $g_\eta$ should be large when $\eta$ is an eigenvalue and small otherwise. Thus, we determine eigenvalues by seeking peaks in the contour plot of the LSM indicator function. We refer to \cite{camano_lackner_monk} for another implementation of this method to determine eigenvalues related to electromagnetic scattering. \par

We now provide a simple example that modified electromagnetic transmission eigenvalues can be detected from simulated far field data and that they shift in response to changes in the permittivity $\epsilon$. For convenience we choose $D$ to be the unit ball in $\mathbb{R}^3$, $B=D$, and $\epsilon$ to be a real constant in $D$, and we choose the wave number as $k=2$ and the number of incident fields as $N_{\text{inc}} = 99$. In this case both the physical scattering problem \eqref{sc} and the auxiliary problem \eqref{aux} may be computed via separation of variables, and we may compute exact eigenvalues in the same manner for the problem \eqref{mit}. We refer to the Appendix for a discussion of this procedure, and we note that the numerical evaluation of the vector spherical harmonics in the series expansions is based on \cite{wang_wang_xie}. Before we continue to an example, we remark that a root-finding algorithm is used to determine the roots of the modified determinant functions given by \eqref{moddet}, which we refer to as the exact eigenvalues. \par

In Figure \ref{fig:shift} we provide an example of both the detection of modified electromagnetic transmission eigenvalues and their shift due to changes in $\epsilon$ for the choices $\gamma = 0.5$ (left column) and $\gamma = 2$ (right column). We have shown a wide range of $\eta$-values in the top row, and in the bottom row we have restricted the plot to a smaller interval in order to highlight the noticeable shift in the most sensitive eigenvalues. The physical and auxiliary electric far field patterns have been computed using separation of variables, and the resulting modified electric far field operator has been subjected to approximatly 2\% multiplicative uniform noise (see \cite{cogar_colton_meng_monk} for details). For both $\gamma=0.5$ and $\gamma=2$ we observe that multiple eigenvalues are detected in this range and that these eigenvalues shift in response to a change from $\epsilon = 2$ to $\epsilon = 1.9$, suggesting that this class of eigenvalues has the potential to be useful in detecting flaws in a material given electromagnetic scattering data. In this simple example we do not see a significant difference between $\gamma=0.5$ and $\gamma=2$; in both cases the most sensitive eigenvalues shift by a comparable amount. However, it has been observed for scalar problems of a similar type that some choices of $\gamma$ can increase this sensitivity to changes in the material, sometimes by an order of magnitude (cf. \cite{cogar,cogar_colton_meng_monk,cogar_colton_monk}). We plan to investigate this question and the broader behavior of these eigenvalues for more complicated domains in a forthcoming manuscript.

\begin{figure}
\begin{subfigure}{.5\textwidth}
\centering
\includegraphics[width=\linewidth]{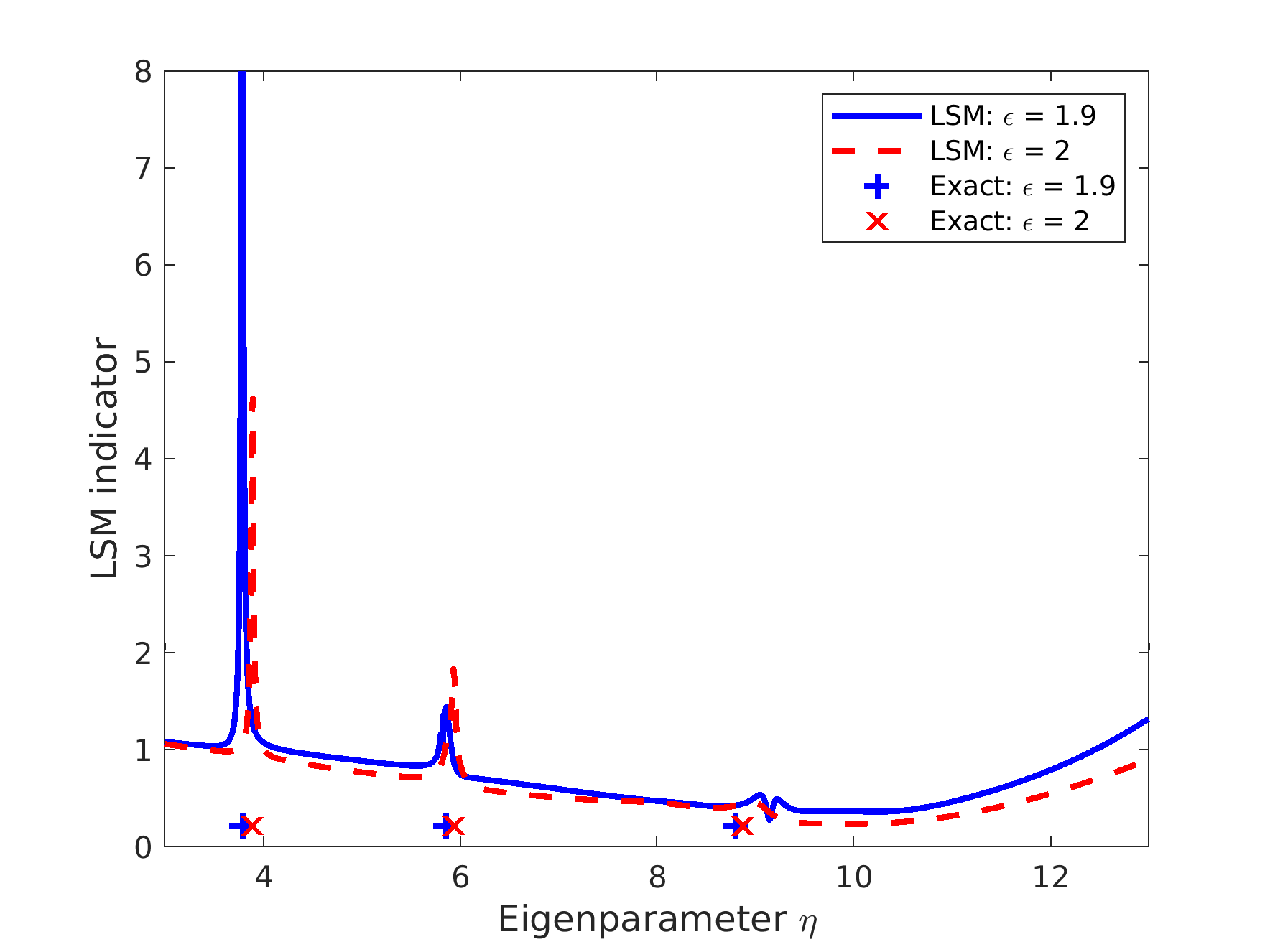}
\caption{$\gamma = 0.5$}
\label{fig:shift_gammap5_wide}
\end{subfigure}
\begin{subfigure}{.5\textwidth}
\centering
\includegraphics[width=\linewidth]{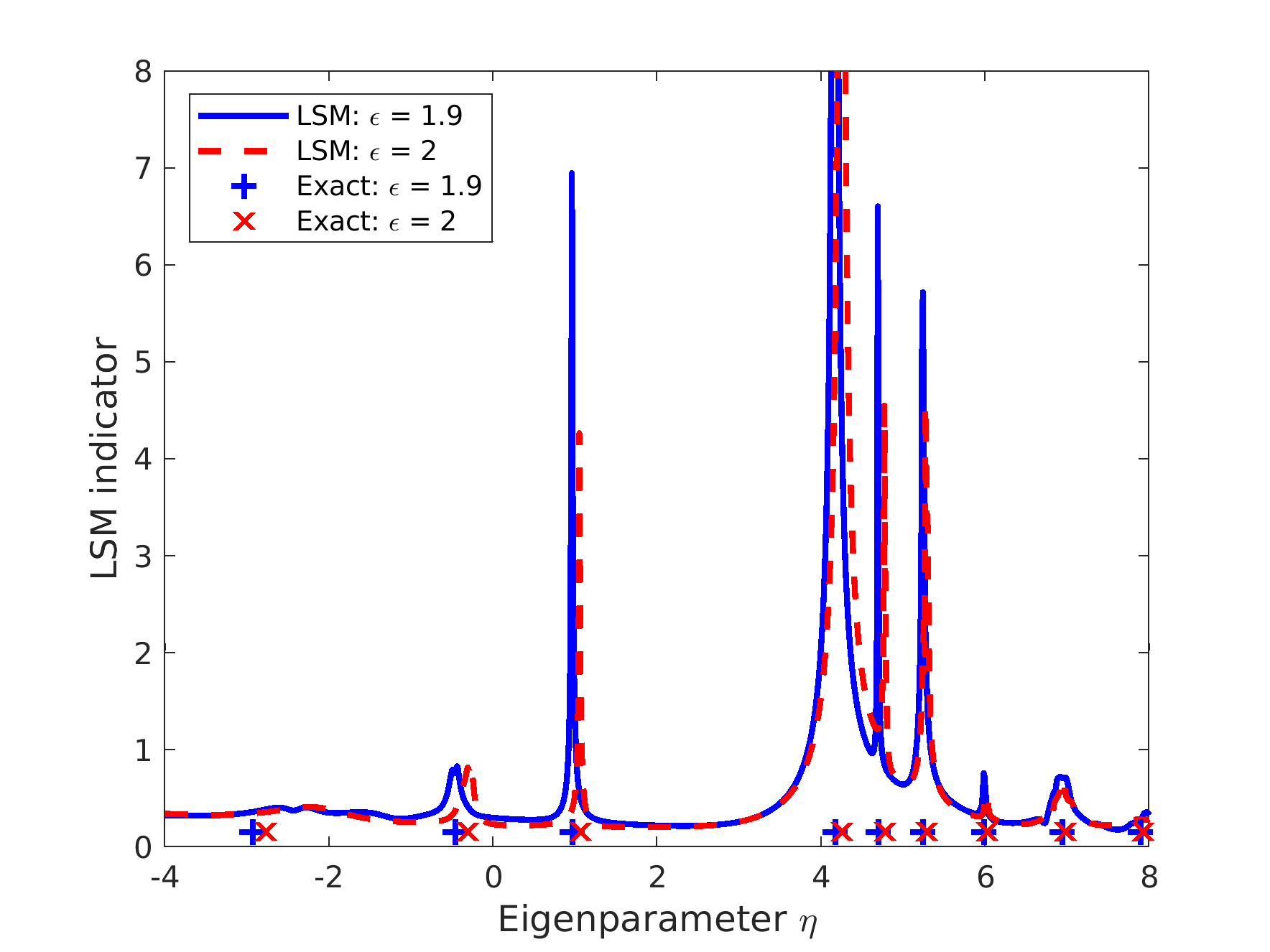}
\caption{$\gamma = 2$}
\label{fig:shift_gamma2_wide}
\end{subfigure}
\begin{subfigure}{.5\textwidth}
\centering
\includegraphics[width=\linewidth]{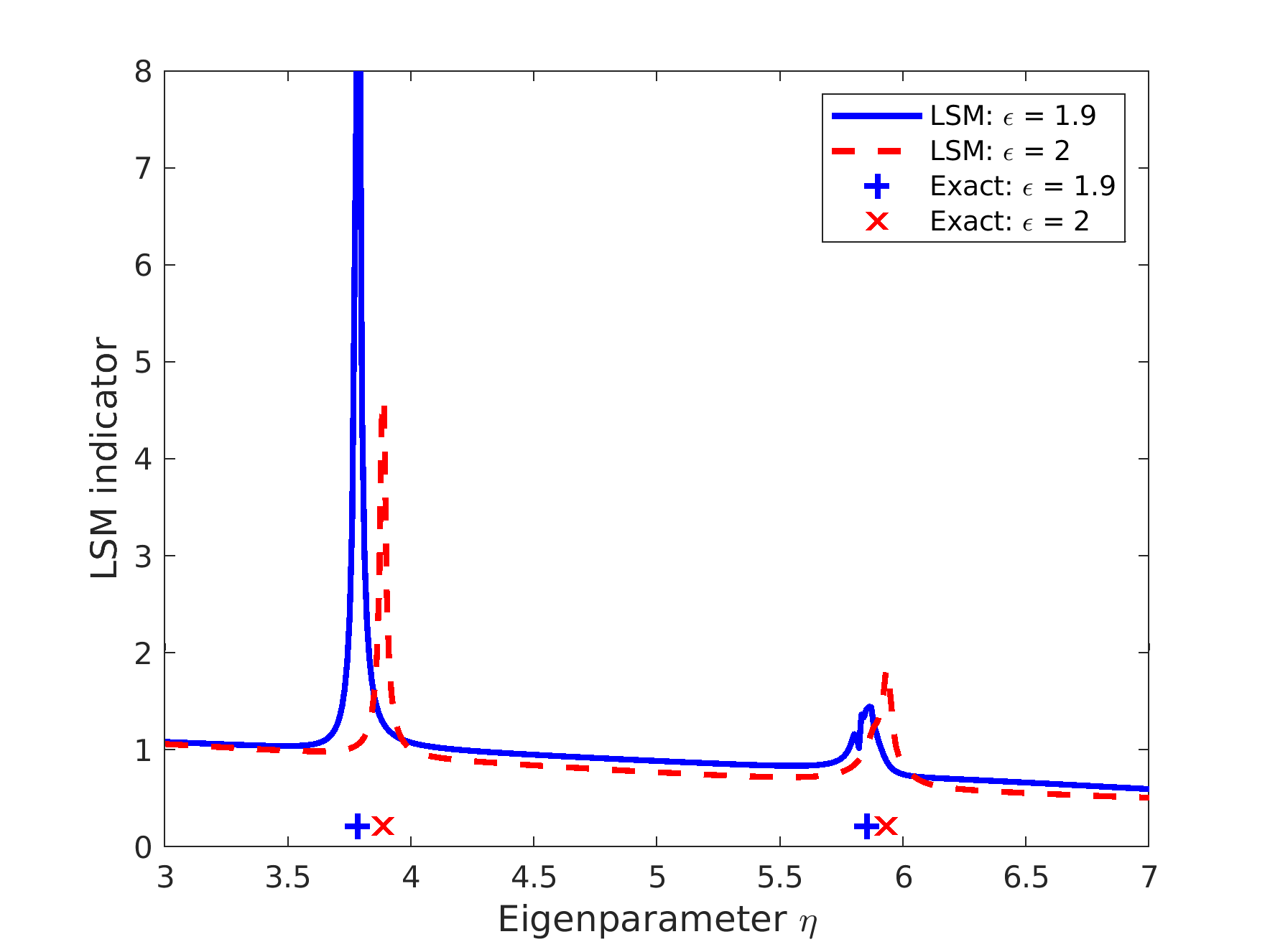}
\caption{$\gamma = 0.5$}
\label{fig:shift_gammap5_zoom}
\end{subfigure}
\begin{subfigure}{.5\textwidth}
\centering
\includegraphics[width=\linewidth]{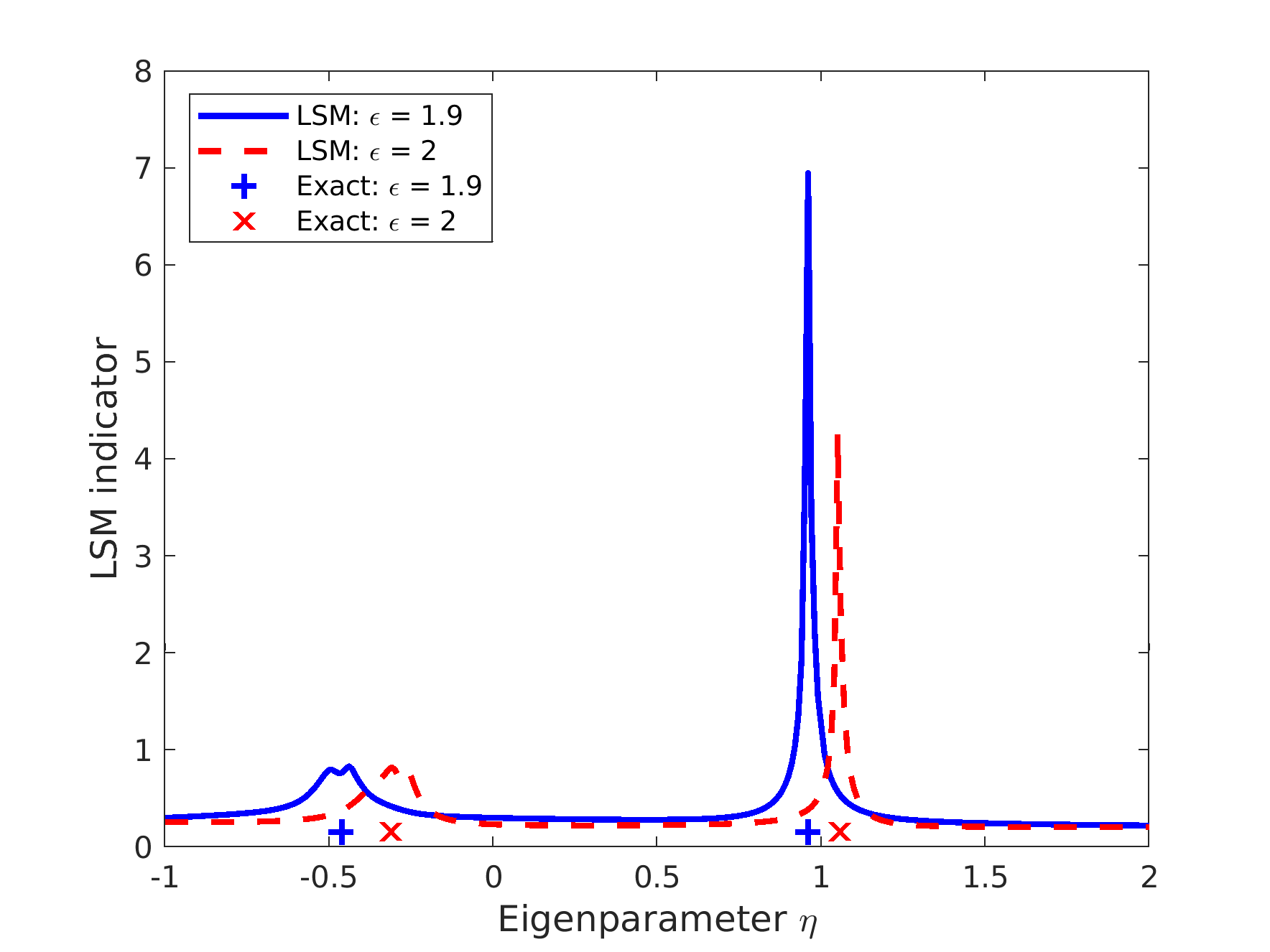}
\caption{$\gamma = 2$}
\label{fig:shift_gamma2_zoom}
\end{subfigure}
\caption{The shift of the eigenvalues due to a change in $\epsilon$ for two different values of $\gamma$. The figures in the top row demonstrate the detection of eigenvalues for $\epsilon = 1.9, 2$ for both $\gamma=0.5$ (left column) and $\gamma=2$ (right column). The figures in the bottom row are identical to the top row but with a narrower interval that highlights the shift in the most sensitive eigenvalues.}
\label{fig:shift}
\end{figure}


\section{Conclusion} \label{sec_conclusion}

We introduced a new eigenvalue problem related to electromagnetic scattering for potential use as a target signature in nondestructive testing of materials. This class of eigenvalues is generated by comparing the measured scattering data for a given medium to an auxiliary problem that depends on a parameter $\eta$ and may be computed independently of the measured scattering data. After showing that this auxiliary problem is well-posed, we established that a nonhomogeneous version of the modified interior transmission problem is of Fredholm type and investigated some properties of the eigenvalues, proving that the eigenvalues are discrete and that infinitely many eigenvalues exist if $\epsilon$ is real-valued. We concluded by showing that the eigenvalues may be determined from the measured electric far field data using the linear sampling method and providing a simple example for scattering by the unit ball with constant permittivity $\epsilon$. \par

Many questions remain open concerning eigenvalue problems generated from modified far field operators, including a general existence theory for the eigenvalues when $\epsilon$ is complex-valued and non-smooth in $B$. Of particular interest for electromagnetic problems of this type is the characterization of when the straightforward electromagnetic analogue of scalar eigenvalue problems possesses the Fredholm property. In the present context, we remarked in Section \ref{sec_introduction} that the eigenvalue problem \eqref{mp} may not possess this property in the range of $\epsilon|_D$, which in the example given consisted of a single point. For variable $\epsilon$ this range would be a larger set, and the solvability of the problem is unknown when $\eta$ falls in this range. However, this question was addressed for electromagnetic Stekloff eigenvalues in \cite{halla2,halla1}, and it is possible that similar techniques may be applied to modified electromagnetic transmission eigenvalues as well. Finally, results on the stability of modified electromagnetic transmission eigenvalues will be presented in a forthcoming manuscript, along with a limited existence result for eigenvalues corresponding to complex-valued $\epsilon$ based on classical perturbation theory from \cite{kato}.
 
\section*{Appendix} \label{app_sov}

We now provide some details about the separation of variables procedure for the auxiliary problem \eqref{aux} and the eigenvalue problem \eqref{mit} in the case where $D$ is chosen to be the unit ball in $\mathbb{R}^3$, $B=D$, and $\epsilon$ is constant in $D$. Separation of variables for the physical scattering problem \eqref{sc} is standard, and we refer to \cite{colton_kress,kirsch_hettlich} for a detailed treatment.  For the auxiliary problem, we assume that $\eta\neq0$ and we replace $\mathbf{E_0}$ in \eqref{aux} with $\mathbf{E}_0 + \eta^{-1}\nabla P$, which results in the equivalent problem of finding $\mathbf{E}_0^s\in \Hcurlloc{\mathbb{R}^3\setminus\overline{B}}$, $\mathbf{E}_0\in \Hcurl{B}$, and $P\in H_*^1(B) $ satisfying
\begin{subequations} \label{eqaux}
\begin{align}
\curl\curl \mathbf{E}_0^s - k^2 \mathbf{E}_0^s &= \mathbf{0} \text{ in } \mathbb{R}^3\setminus\overline{B}, \label{eqaux1} \\
\curl\gamma^{-1}\curl \mathbf{E}_0 - k^2\eta \mathbf{E}_0 &= \mathbf{0} \text{ in } B, \label{eqaux2} \\
\Delta P &= 0 \text{ in } B, \label{eqaux3} \\
\eta^{-1}\normal{P} + \un\cdot \mathbf{E}_0 &= 0 \text{ on } \partial B, \label{eqaux4} \\
\un\times\mathbf{E}_0 + \eta^{-1}\un\times\nabla P - \un\times\mathbf{E}_0^s &= \un\times \mathbf{E}^i \text{ on } \partial B, \label{eqaux5} \\
\un\times\gamma^{-1}\curl \mathbf{E}_0 - \un\times\curl \mathbf{E}_0^s &= \un\times\curl \mathbf{E}^i \text{ on } \partial B, \label{eqaux6} \\
\mathclap{\lim_{r\to\infty} \left(\curl \mathbf{E}_0^s\times\mathbf{x} - ikr\mathbf{E}_0^s\right) = 0.} \label{eqaux7}
\end{align}
\end{subequations}
In this form we may apply the standard approach for Maxwell's equations along with expansion methods for Laplace's equation found in \cite{kirsch_hettlich}, in which the scattered and total fields are of the form
\begin{align*}
\mathbf{E}_0^s(\mathbf{x}) &= \sum_{n=1}^\infty \sum_{m=-n}^n \Biggr[ \alpha_n^m \frac{\sqrt{n(n+1)}}{r} h_n^{(1)}(kr) Y_n^m(\mathbf{\hat{x}})\mathbf{\hat{x}} \\
&\hspace{8em} + \alpha_n^m \frac{(rh_n^{(1)}(kr))'}{r} \mathbf{U}_n^m(\hat{x}) + \beta_n^m h_n^{(1)}(kr) \mathbf{V}_n^m(\hat{x}) \Biggr], \; \mathbf{x}\in\mathbb{R}^3\setminus\overline{B}, \\
\mathbf{E}_0(\mathbf{x}) &= \sum_{n=1}^\infty \sum_{m=-n}^n \Biggr[ \delta_n^m \frac{\sqrt{n(n+1)}}{r} j_n(k\sqrt{\gamma\eta}r) Y_n^m(\mathbf{\hat{x}})\mathbf{\hat{x}} \\
&\hspace{8em} + \delta_n^m \frac{(rj_n(k\sqrt{\gamma\eta}r))'}{r} \mathbf{U}_n^m(\hat{x}) + \varphi_n^m j_n(k\sqrt{\gamma\eta}r) \mathbf{V}_n^m(\hat{x}) \Biggr], \; \mathbf{x}\in B,
\end{align*}
and the scalar field $P$ is of the form
\begin{equation*}
P(\mathbf{x}) = \sum_{n=1}^\infty \sum_{m=-n}^n p_n^m r^n Y_m^n(\mathbf{\hat{x}}), \; \mathbf{x}\in B.
\end{equation*}
Here we have denoted by $\mathbf{U}_n^m$ and $\mathbf{V}_n^m$ the vector spherical harmonics
\begin{equation*}
\mathbf{U}_n^m(\hat{x}) = \frac{1}{\sqrt{n(n+1)}} \nabla_{\mathbb{S}^2} Y_n^m(\mathbf{\hat{x}}), \; \mathbf{V}_n^m(\mathbf{\hat{x}}) = \mathbf{\hat{x}}\times\mathbf{U}_n^m(\mathbf{\hat{x}}).
\end{equation*}
Applying the boundary conditions \eqref{eqaux4}--\eqref{eqaux6} (in a similar manner to \cite{camano_lackner_monk}) now implies that the coefficients satisfy
\begin{gather}
\begin{split} \label{system0}
\left( \begin{array}{ccccc} (rh_n^{(1)}(kr))'|_{r=1} & 0 & -(rj_n(k\sqrt{\gamma\eta}r))'|_{r=1} & 0 & -\eta^{-1} \sqrt{n(n+1)} \\
0 & h_n^{(1)}(k) & 0 & -j_n(k\sqrt{\gamma\eta}) & 0 \\
0 & (rh_n^{(1)}(kr))'|_{r=1} & 0 & -\gamma^{-1} (rj_n(k\sqrt{\gamma\eta}r))'|_{r=1} & 0 \\
h_n^{(1)}(k) & 0 & -\eta j_n(k\sqrt{\gamma\eta}) & 0 & 0 \\
0 & 0 & \sqrt{n(n+1)} j_n(k\sqrt{\gamma\eta}) & 0 & \eta^{-1} n \end{array} \right) \\
\times \left( \begin{array}{c} \alpha_n^m \\ \beta_n^m \\ \delta_n^m \\ \varphi_n^m \\ p_n^m \end{array} \right) = \left( \begin{array}{c} -a_n^m (rj_n(kr))'|_{r=1} \\ -b_n^m j_n(k) \\ -b_n^m (rj_n(kr))'|_{r=1} \\ -a_n^m j_n(k) \\ 0 \end{array} \right),
\end{split}
\end{gather}
where the coefficients $a_n$ and $b_n$ are chosen such that the incident field $\mathbf{E}^i$ may be expanded as
\begin{equation*}
\mathbf{E}^i(\mathbf{x}) = \sum_{n=1}^\infty \sum_{m=-n}^n \left[ a_n^m \frac{\sqrt{n(n+1)}}{r} j_n(kr) Y_n^m(\mathbf{\hat{x}})\mathbf{\hat{x}} + a_n^m \frac{(rj_n(kr))'}{r} \mathbf{U}_n^m(\hat{x}) + b_n^m j_n(kr) \mathbf{V}_n^m(\hat{x}) \right], \; \mathbf{x}\in\mathbb{R}^3.
\end{equation*}
In this case of a plane wave incident field with propagation direction $\mathbf{d}$ and polarization vector $\mathbf{p}$, these coefficients are given by
\begin{equation*}
a_n^m = -\frac{4\pi i^{n+1}}{k} \overline{\mathbf{U}_n^m(\mathbf{d})}^T \mathbf{p}, \; b_n^m =4\pi i^n \overline{\mathbf{V}_n^m(\mathbf{d})}^T \mathbf{p},
\end{equation*}
as a result of the Jacobi-Anger expansion (cf. \cite{colton_kress}). With the coefficients satisfying \eqref{system0}, the auxiliary far field pattern may be constructed by the formula
\begin{equation} \label{FFP}
\mathbf{E}_{0,\infty}(\mathbf{\hat{x}},\mathbf{d};\mathbf{p}) = -\frac{1}{k} \sum_{n=1}^\infty \frac{1}{i^{n+1}} \sum_{m=-n}^n \left[ \alpha_n^m \sqrt{n(n+1)}\mathbf{V}_n^m(\mathbf{\hat{x}}) + ik\beta_n^m \sqrt{n(n+1)}\mathbf{U}_n^m(\mathbf{\hat{x}}) \right].
\end{equation}
Following this same procedure for the modified electromagnetic transmission eigenvalue problem \eqref{mit} implies that $\eta\neq0$ is an eigenvalue if and only if it is a root of one of the determinant functions
\begin{subequations} \label{moddet}
\begin{align}
\tilde{d}_n^{(a)}(\eta) &= \left(1-\gamma^{-1}\right)j_n(k\sqrt{\epsilon})j_n(k\sqrt{\gamma\eta}) + k\sqrt{\epsilon} j_n'(k\sqrt{\epsilon})j_n(k\sqrt{\gamma\eta}) - k\sqrt{\gamma^{-1}\eta} j_n(k\sqrt{\epsilon}) j_n'(k\sqrt{\gamma\eta}), \label{moddet_a} \\
\tilde{d}_n^{(b)}(\eta) &= (\eta+n\epsilon)j_n(k\sqrt{\epsilon})j_n(k\sqrt{\gamma\eta}) + k\eta\sqrt{\epsilon} j_n'(k\sqrt{\epsilon})j_n(k\sqrt{\gamma\eta}) - k\epsilon\sqrt{\gamma\eta} j_n(k\sqrt{\epsilon}) j_n'(k\sqrt{\gamma\eta}). \label{moddet_b}
\end{align}
\end{subequations}
Compared to the standard determinant functions given by \eqref{det}, we see that the only difference is that $(\eta-\epsilon)$ in \eqref{det_b} is replaced with $(\eta+n\epsilon)$ in \eqref{moddet_b}. A computational study using the parameters from Figure \ref{fig_sov1} shows that the roots of these families of equations do not accumulate at the constant value of $\epsilon$.


\FloatBarrier

\bibliographystyle{siamplain}
\bibliography{myreferences2}

\end{document}